\documentclass[12pt]{amsart}

\usepackage{amssymb,amsthm,amsfonts,latexsym}
\usepackage{amsmath}
\usepackage{color}
\usepackage{enumitem}
\usepackage{hyperref}
\usepackage{pxfonts} 

\input{xy}
\xyoption{all}
\xyoption{poly}
\usepackage[all]{xy}
\allowdisplaybreaks

\theoremstyle{plain}
\newtheorem{theorem}{Theorem}[section]
\newtheorem*{theorem*}{Theorem}
\newtheorem{lemma}[theorem]{Lemma}
\newtheorem{corollary}[theorem]{Corollary}
\newtheorem*{corollary*}{Corollary}
\newtheorem{proposition}[theorem]{Proposition}
\newtheorem{conjecture}[theorem]{Conjecture}

\theoremstyle{definition}
\newtheorem{example}[theorem]{Example}
\newtheorem{definition}[theorem]{Definition}

\theoremstyle{remark}
\newtheorem{remark}[theorem]{Remark}

\numberwithin{equation}{section}

 \newcommand\KK{{\mathbb{K}}}
 
 \def\QQ{\mathbb{Q}}
 
 \newcommand\ZZ{{\mathbb{Z}}}
 
 \def\A{\mathcal{A}}
 \def\B{{\mathcal B}}

 \def\F{{\mathcal F}}

 \def\N{{\mathcal N}}

 \def\S{{\mathcal S}}

 \def\Lk{{\text{Lk}}}
 
 \def\St{\text{St}}  
 
 \def\op{\mathrm{op}}

 \DeclareMathOperator\GL{GL}
 \DeclareMathOperator\PSL{PSL}
 \DeclareMathOperator\PGL{PGL}

 \DeclareMathOperator{\PSU}{PSU}
 
 \DeclareMathOperator{\PSp}{PSp}
 
 \DeclareMathOperator{\POmega}{P\Omega}

 \DeclareMathOperator{\Sym}{Sym}
 \DeclareMathOperator{\Alt}{Alt}

 \DeclareMathOperator\Lie{Lie}
 
 \DeclareMathOperator\Aut{Aut}
 
 \DeclareMathOperator\Out{Out}
 \DeclareMathOperator\Inn{Inn}
 \DeclareMathOperator\Outdiag{Out diag}
 \DeclareMathOperator\Inndiag{Inn diag}

 \DeclareMathOperator\Ker{Ker}

 \newcommand{\tq}{\mathrel{{\ensuremath{\: : \: }}}}

 \DeclareMathOperator\Stab{Stab}
 
 \DeclareMathOperator\Ind{Ind}

 \def\Sp{\mathcal{S}_p} 
 \def\Ap{\mathcal{A}_p}
 \def\Bp{\mathcal{B}_p}

 \def\groupiso{\cong}

 \newcommand\gen[1]{\left\langle#1\right\rangle}

\renewcommand{\gen}[1]{\langle #1 \rangle}
\def\Frob{\mathbb{F}}
\def\GG{\mathbb{G}}

\def\Lie{\mathrm{Lie}}
\def\PSp{\mathrm{PSp}}
\def\name{\mathrm{SCNL}}

\begin{document}

\title{Spherical $p$-group complexes arising from finite groups of Lie type}

\author{\small Kevin I. Piterman}
\address{Fachbereich Mathematik und Informatik, Philipps-Universit{\"a}t Marburg, 35032 Marburg, Germany; and Department of Mathematics and Data Science, Vrije Universiteit Brussel, 1050 Brussels, Belgium}
\email{kevin.piterman@vub.be}

\keywords{$p$-Subgroup Complexes; Tits Buildings; Groups of Lie Type}
\subjclass[2020]{05E18, 06A11, 20D30, 20E42, 20G40, 55U05}

\begin{abstract}
We show that the $p$-group complex of a finite group $G$ is homotopy equivalent to a wedge of spheres of dimension at most $n$ if $G$ contains a self-centralising normal subgroup $H$ which is isomorphic to a group of Lie type and Lie rank $n$ in characteristic $p$.
If in addition, every order-$p$ element of $G$ induces an inner or field automorphism on $H$, the $p$-group complex of $G$ is $G$-homotopy equivalent to a spherical complex obtained from the Tits building of $H$.

We also prove that the reduced Euler characteristic of the $p$-group complex of a finite group $G$ is non-zero if $G$ has trivial $p$-core and $H$ is a self-centralising normal subgroup of $G$ which is a group of Lie type (in any characteristic), except possibly when $p=2$ and $H=A_n(4^a)$ ($n\geq 2$) or $E_6(4^a)$.
In particular, we conclude that the Euler characteristic of the $p$-group complex of an almost simple group does not vanish for $p\geq 7$.
\end{abstract}

\maketitle

\section{Introduction}

For a finite group $G$ and a prime $p$, the Quillen poset $\Ap(G)$ consists of the non-trivial elementary abelian $p$-subgroups of $G$ ordered by inclusion.
D. Quillen introduced this poset in 1978, and established many connections between homotopical properties of its order complex (which we shall call the $p$-group complex of $G$) and intrinsic algebraic properties of $G$ (see \cite{Qui78}).
For instance, when $G$ is a finite group of Lie type in characteristic $p$, the poset $\Ap(G)$ is $G$-homotopy equivalent to the poset of proper parabolic subgroups of $G$, and therefore its homology, which is concentrated in a single degree, gives rise to the Steinberg module of $G$.

On the other hand, for an arbitrary finite group $G$, the fixed point subposet $\Ap(G)^g$ is contractible for any element $g\in G$ of order divisible by $p$.
This implies that, although in general the homology of $\Ap(G)$ might not be concentrated in a single degree, the virtual character associated with the Lefschetz module of $\Ap(G)$ (viewed, for example, as an element in the Green ring of $\QQ G$-modules) vanishes at elements of order divisible by $p$, recovering one of the crucial features of the Steinberg character. 
Motivated by these observations, one usually regards the Lefschetz module of $\Ap(G)$ as the ``mod $p$" analogue of the Steinberg module for an arbitrary finite group $G$ (cf. \cite{Webb}).

In general, the homotopy type of the $p$-group complexes is far from being completely understood.
The first computations by Quillen \cite{Qui78} suggest that $\Ap(G)$ has the homotopy type of a wedge of spheres of possibly different dimensions.
In fact, if $G$ is a solvable group, J. Pulkus and V. Welker \cite{PW}  showed that there is a wedge decomposition for $\Ap(G)$ such that if $N$ is a normal $p'$-subgroup of $G$, and $\Ap(G/N)_{>A}$ is a wedge of spheres for all $A \in \Ap(G/N)$, then $\Ap(G)$ is a wedge of spheres.
Thus, the determination of the homotopy type of the $p$-group complexes of solvable groups roughly boils down to a question about the upper-intervals.
However, for non-solvable groups, $\Ap(G)$ might not be a wedge of spheres, and J. Shareshian \cite{Shareshian} exhibited the first example: There is torsion in the homology of the Quillen poset of the alternating group on $13$ letters for $p=3$, so this complex cannot have the homotopy type of a wedge of spheres.
Furthermore, it was shown later that the failure of being a wedge of spheres can also arise from the fundamental group: in \cite{MP}, E.G. Minian and K.I. Piterman prove that the fundamental group of the Quillen poset of the alternating group on $10$ letters for $p=3$ is not a free group, but its homology is abelian free.
In this case, the homology cannot detect the failure to be a wedge of spheres, indicating that homology groups generally do not provide enough information to determine the homotopy type of the $p$-group complexes.
It is worth noting that these examples arise from simple groups.

\smallskip

In this article, we describe the homotopy type of the $p$-group complex for finite groups that contain suitable subgroups of Lie type in characteristic $p$.

Given a group $G$ and a prime $p$, we say that $H$ is an \textit{$\name_p$}-subgroup of $G$ if $H$ is a normal and self-centralising subgroup of $G$ (i.e., if $g\in G$ centralises $H$ then $g\in H$) such that $H$ is isomorphic to a finite group of Lie type in characteristic $p$.
Here by a finite group of Lie type in characteristic $p$ we mean $H = O^{p'}(\GG^F)$, where $O^{p'}(K)$ denotes the smallest normal subgroup of index prime to $p$ of a finite group $K$, and $\GG^F$ is the fixed point group of a Steinberg endomorphism $F$ of a simple algebraic group $\GG$ over an algebraically closed field of characteristic $p$.
The Lie rank of $H$ is the rank of the underlying building. 
Indeed, if such an $\name_p$-subgroup $H$ exists, then it must be unique (cf. Lemma \ref{lm:uniqueLiep}), and if further $H$ is quasisimple then $H = F^*(G)$, the generalised Fitting subgroup of $G$.

Our main result can be summarised as follows:

\begin{theorem}
\label{thm:introMain}
Let $p$ be a prime and $G$ a finite group with an $\name_p$-subgroup $H$ of Lie rank $n$.
Then $\Ap(G)$ is homotopy equivalent to a wedge of spheres of dimension at most $n$.

Moreover, if $\{E\in \Ap(G)\tq E\cap H = 1, O_p(C_H(E)) = 1\}$ is non-empty and consists only of cyclic subgroups inducing field automorphisms on $H$, then the order complex of the Bouc poset $\Bp(\Omega_1(G))$ is spherical of dimension $n$.
\end{theorem}

Here a simplicial complex of dimension $n$ is said to be spherical if it is homotopy equivalent to a wedge of spheres of dimension $n$.
Also recall that $\Omega_1(G)$ is the subgroup generated by the order-$p$ elements of $G$, so $\Ap(G) = \Ap(\Omega_1(G))$.
Indeed the Quillen poset $\Ap(G)$ is $G$-homotopy equivalent to the Bouc poset $\Bp(G)$ consisting of the radical $p$-subgroups of $G$, i.e., the non-trivial $p$-subgroups $R\leq G$ such that $O_p(N_G(R)) = R$, where $O_p(N)$ is the $p$-core of $N$, and $N_G(R)$ is the normaliser of $R$ in $G$ (see \cite{Bouc}).
Finally, $C_H(E)$ is the centraliser of $E$ in $H$.
See Theorem \ref{thm:main} for the general statement, which additionally contains an explicit description of the spheres appearing in the wedge decomposition of $\Ap(G)$.
The conclusions on the sphericity of $\Bp(\Omega_1(G))$ are contained in Corollary \ref{coro:sphericalBp}.

Moreover, Theorem \ref{thm:main} yields a formula for the reduced homology of $\Ap(G)$ with rational coefficients, denoted by $\widetilde{H}_*(\Ap(G),\QQ)$, which we briefly describe in the following corollary.
If $G$ contains an $\name_p$-subgroup $H$, then we define
\begin{align*}
    \F_f & = \{ E\in \Ap(G) \tq |E|=p \text{ and $E$ induces field automorphisms on }H\},\\
    \F_g & = \{ E\in \Ap(G) \tq |E|=p,\  O_p(C_{H}(E)) = 1, \ E\notin \F_f\}.
\end{align*}
By uniqueness of $H$ (if it exists), these sets depend exclusively on $G$ and $p$.
Note that $\F_f,\F_g$ are invariant under the conjugation action of $G$, and we write $\overline{E}$ for the conjugacy class of $E\in \F_f\cup \F_g$, and $\F_f/G$, $\F_g/G$ for the conjugacy class sets.

The elements of $\F_g$ induce ``graph" or ``graph-field" automorphisms on $H$.
In any case, if $E\in \F_f\cup\F_g$ then $O^{p'}(C_H(E))$ is a group of Lie type in characteristic $p$ of a certain Lie rank $m_E$.
For $E\in \F_g$, write $m_E^* = m_E$ if $E$ centralises some element of $\F_f$, and $m_E^*=m_E-1$ otherwise.
Denote by $G_{df}$ the pre-image by the natural map $G\to\Aut(H)$ of the subgroup $\Inndiag(H)\Phi_{H}$, where $\Inndiag(H)$ is the group of inner-diagonal automorphisms of $H$ and $\Phi_{H}$ is the full group of field automorphisms of $H$ (in the Steinberg sense \cite[Section 10]{Steinberg}, see Section \ref{sec:algebraicgroups} for more details).
Note that $G_{df}$ is a normal subgroup of $G$.

\begin{corollary}
Let $G$ be a finite group, $p$ a prime, and suppose that $H$ is an $\name_p$-subgroup of $G$ of Lie rank $n$.
Then, for $m\geq 0$ we have an isomorphism of $G$-modules:
\[ \widetilde{H}_m(\Ap(G),\QQ) \cong \widetilde{H}_m(\Ap(G_{df}),\QQ) \oplus \bigoplus_{\overline{E}\in \F_g/G} \Ind_{N_G(E)}^G\big( \widetilde{H}_{m-1}(\Ap(C_{G_{df}}(E),\QQ) \big),\]
where $\widetilde{H}_{m-1}(\Ap(C_{G_{df}}(E),\QQ) \big)\neq 0$ if and only if $m-1 = m_E^*$.
Moreover, if $\F_f = \emptyset$ then
\[ \widetilde{H}_*( \Ap(G_{df}),\ZZ) = \widetilde{H}_{n-1}(\Ap(G_{df}),\ZZ) = \widetilde{H}_{n-1}( \Ap(H),\ZZ),\]
and if $\F_f\neq\emptyset$ then $\widetilde{H}_*( \Ap(G_{df}), \ZZ) = \widetilde{H}_n( \Ap(G_{df}),\ZZ)$ is the kernel of the surjective $G$-equivariant map
\[ \bigoplus_{\overline{E}\in \F_f/G} \Ind_{N_G(E)}^G\big(\widetilde{H}_{n-1}( \Ap(C_H(E)) ,\ZZ )\big) \overset{i}{\longrightarrow} \widetilde{H}_{n-1}( \Ap(H),\ZZ),\]
where $i$ is induced by the inclusions $\Ap(C_H(E))\hookrightarrow \Ap(H)$.
\end{corollary}

See Theorems \ref{thm:fieldCase} and \ref{thm:main}.
From this description, one can write down an expression for the Lefschetz module of $\Ap(G)$ in terms of the Steinberg modules of $H$ and the centralisers of outer automorphisms of order $p$ of $H$ included in $G$.
For instance, for the homology of $\Ap(G_{df})$, if $\St_p(H)$ denotes the Steinberg module of $H$ over the rationals then
\[ \widetilde{H}_n( \Ap(G_{df}),\QQ) = \left( \bigoplus_{\overline{E}\in \F_f/G} \Ind_{N_G(E)}^G( \St_p(C_H(E)) ) \right) \ \big/ \ \St_p(H). \]

As a corollary of the wedge decomposition for $\Ap(G)$, we also get:

\begin{corollary}
Let $G$ be a finite group, $p$ a prime, and suppose that $H$ is an $\name_p$-subgroup of $G$.
Then $\widetilde{\chi}(\Ap(G)) \neq 0$, except possibly if $p=2$, $H=A_n(4^a)$ or $E_6(4^a)$ and both $\F_f$ and $\F_g$ are non-empty.
\end{corollary}

Here $\widetilde{\chi}(\Ap(G))$ denotes the reduced Euler characteristic of $\Ap(G)$.
Indeed, the excluded cases in the above corollary may be ruled out by performing some extra (but lengthier) computations according to the possibilities of $G$ and the formula we get for the homology of $\Ap(G)$.
See Corollary \ref{coro:euler}.


\bigskip

The results on the homotopy type of the Quillen complex of such a family of groups $G$ are obtained by basically showing that $\Ap(G)$ is homotopy equivalent to the Tits building of the $\name_p$-subgroup $H$ after adding new vertices with highly-connected links, which are basically the Tits buildings of the centralisers in $H$ of the elements in $\F_f\cup \F_g$.
We provide preliminaries on more theoretic homotopy results in Section \ref{sec:simplicialComplexes}, and an Appendix \ref{sec:appendix} with some classical theorems of algebraic topology to compute homotopy types.
In Section \ref{sec:psubgroups}, we set the notation related to finite groups, and exhibit homotopy equivalences between different $p$-subgroup complexes, not always obtained from standard $p$-subgroup posets.
Then, in Section \ref{sec:buildings}, we briefly state some intermediate consequences of these methods on buildings.
In Section \ref{sec:algebraicgroups}, we discuss preliminaries on finite groups of Lie type, adopting mainly the language of \cite{MalleTesterman} but also using \cite{GLS98}, and prove the main result of this article Theorem \ref{thm:main}.
Finally, Section \ref{sec:euler} contains computations with the Euler characteristic of the Quillen poset, and hence the proof of Corollary \ref{coro:euler}.
In view of these results, we also conjecture that $\widetilde{\chi}(\Ap(G))\neq 0$ for any almost simple group $G$ and prime $p$ (this is a stronger reformulation of the original Quillen's conjecture which states that $\Ap(G)$ is non-contractible if $O_p(G) = 1$, see \cite{Qui78}).
To support this conjecture, we establish the non-vanishing of the reduced Euler characteristic of the Quillen poset of an almost simple group $G$ such that none of the following cases holds: $F^*(G)$ is either a sporadic group with $p=2,3,5$, an alternating group with $p=2,3$, or $A_n(4^a)$ ($n\geq 2$), $E_6(4^a)$ with $p=2$ and $\F_f,\F_g$ non-empty.
The completeness of this list depends on the classification of the finite simple groups.
See Corollary \ref{coro:eulerAlmostSimple}.

\smallskip

Computer calculations were performed in \cite{Gap} with package \cite{FPSC}.

\bigskip

\textbf{Acknowledgements.}
I am very grateful to Gunter Malle for the many discussions on algebraic groups that enriched my understanding of this topic and for his numerous comments on an earlier version of this article that substantially improved its presentation.
I would also like to thank Stefan Witzel for their insightful conversations on buildings that motivated this work, and Volkmar Welker for his comments on a preliminary version.

The author was supported by a Postdoctoral Fellowship granted by the Alexander von Humboldt Stiftung, the FWO grant 12K1223N, the FWO Senior Research Project G004124N, and project OZR3762 of Vrije Universiteit Brussel.

\section{Preliminaries on simplicial complexes}
\label{sec:simplicialComplexes}

In this section, we set some notation and provide results on the homotopy type of particular families of simplicial complexes.
These results mostly rely on classical theorems of algebraic topology, which we include in the Appendix \ref{sec:appendix}.

Suppose $K$ and $L$ are (abstract) simplicial complexes.
We denote by $K*L$ the simplicial join of $K$ and $L$.
Recall this is a simplicial complex whose vertex set equals the disjoint union of the set of vertices of $K$ and $L$, and there is a simplex $\sigma\cup\tau$ for every $\sigma\in K$ and $\tau\in L$ (including the empty simplex).
If $|K|$ denotes the geometric realisation of $K$, then $|K*L|$ is homeomorphic to the classical topological join $|K|*|L|$.
We will use the symbol $\simeq$ to denote homotopy equivalences between posets, complexes or topological spaces.
The symbol $\cong$ will be mostly used to denote isomorphisms.
A subscript $\simeq_G$ or $\cong_G$ indicates $G$-equivariant homotopy equivalence or $G$-isomorphism respectively, where $G$ is a group.

Suppose that $L$ is a subcomplex of $K$, which we write $L\leq K$.
We say that $L$ is a full subcomplex of $K$ if every simplex of $K$ whose vertices lie $L$ is already contained in $L$.
We write $K\setminus L$ for the full subcomplex of $K$ whose simplices do not contain vertices in $L$.
If $\sigma \in K$ is a simplex, the link of $\sigma$ in $L$ is the subcomplex $\Lk_L(\sigma) = \{\tau\in L\tq \tau\cup\sigma\in K, \  \tau\cap\sigma=\emptyset\}$.
The star of $\sigma$ in $L$ is the subcomplex $\St_L(\sigma)= \{\tau\in L\tq \tau\cup\sigma\in K\}$.
Note that $\St_K(\sigma) = \Lk_K(\sigma)*\sigma = \St_{K\setminus \sigma}(\sigma)*\sigma$, where here we regard $\sigma$ as a simplicial complex whose simplices are the faces of $\sigma$.

If $G$ is a group, we say that $K$ is a $G$-complex if $G$ acts on $K$ by simplicial automorphisms.
We denote by $\Stab_G(\sigma)$ the stabiliser in $G$ of a simplex $\sigma\in K$, and by $K^G$ the subcomplex of pointwise fixed simplices of $K$.

We will need the following result that describes the homotopy type of a simplicial complex $K$ from a full subcomplex $L$ such that $K\setminus L$ is a discrete complex, that is, its simplices have dimension at most $0$.

\begin{theorem}
\label{thm:MVsequence}
Let $K$ be a $G$-complex and let $L$ be a full $G$-invariant subcomplex such that $K\setminus L$ is discrete.
Denote by $V$ the set of vertices of $K\setminus L$.
Then we have a pushout in the geometric realisations:
\begin{equation}
\label{eq:pushoutKL}
\begin{aligned}
 \xymatrix{
\coprod_{v\in V} |\Lk_L(v)| \ar@{^(->}[d] \ar[r] & |L| \ar[d]\\
\coprod_{v\in V} |\Lk_L(v)*\{v\}| \ar[r] & |K|
}
\end{aligned}
\end{equation}
Moreover, the following hold:
\begin{enumerate}
    \item There is a long exact sequence in homology with $G$-equivariant maps
    \[ \ldots \to \widetilde{H}_{m+1}(K,\ZZ) \to \bigoplus_{v\in V} \widetilde{H}_m( \Lk_L(v),\ZZ)  \overset{j}{\longrightarrow} \widetilde{H}_m(L,\ZZ) \overset{i}{\longrightarrow} \widetilde{H}_m(K,\ZZ)\to \ldots \]
    where the maps $i,j$ are induced by the corresponding inclusions.

    \item If for each vertex $v\in V$ the inclusion $\Lk_L(v)\hookrightarrow L$ is homotopic to a constant map on a vertex $v^*\in L$, then $K$ is homotopy equivalent to the wedge
    \[ K \simeq L \vee \bigvee_{v\in V} \Lk_L(v)* S_v\]
    where $S_v$ is the discrete complex on vertices $v$ and $v^*$, and $v^*\in L$ is identified with $v^{*} \in S_v \leq \Lk_L(v)*S_v$.

    Moreover, if $G$ is finite, $R$ is a unit ring, and $RG$ is semisimple, then we have an isomorphism of $RG$-modules
    \[ \widetilde{H}_m(K,R) \cong \widetilde{H}_m(L,R) \oplus \bigoplus_{\overline{v}\in V/G} \Ind_{\Stab_G(v)}^G \big( \, \widetilde{H}_{m-1}(\Lk_L(v),R)\, \big),\]
    for all $m\geq 0$.    
    \item If $\Lk_L(v)$ is $\Stab_G(v)$-contractible for all $v \in V$ then $K\simeq_G L$.
\end{enumerate}
\end{theorem}

\begin{proof}
The pushout in Eq. (\ref{eq:pushoutKL}) is a standard construction.
Note that for $v\in V$ we have $\St_K(v) = \Lk_L(v) * \{v\}$, i.e., $\Lk_K(v) = \Lk_L(v)$.

The long exact sequence in item (1) follows by applying the Mayer-Vietoris Theorem to the cover of $|K|$ by the open sets $A = \bigcup_{v\in V} |\St_K(v)|\setminus |\Lk_K(v)|$, and $B = |K| \setminus V$.
Note that the inclusion $|L|\hookrightarrow B$ is a homotopy equivalence.
Moreover, $A$ is the disjoint union of the open sets $|\St_K(v)|\setminus |\Lk_K(v)|$, for $v\in V$, and each of these open sets is contractible.
This sequence can also be obtained by suitably applying Theorem 2.5 of \cite{SW} to the face poset of $K$ with $M$ there being our set of vertices $V$, where the equivariant property of the maps is also justified.

Item (2) follows by noting that in the pushout of Eq. (\ref{eq:pushoutKL}), we can change the arrow $\coprod_{v\in V} |\Lk_L(v)| \to |L|$ by the map induced by the constant maps $x\in |\Lk_L(v)| \mapsto v^*$, for $v\in V$, and since these maps are homotopy equivalent, the pushouts are homotopy equivalent (see Proposition 0.18 of \cite{Hatcher}).

The $G$-module isomorphism follows from the long exact sequence in item (1) with coefficients in $R$ since the zero maps $\widetilde{H}_*(\Lk_L(v),R) \to \widetilde{H}_*(L,R)$ give rise to short exact sequences and they all $R G$-split by semisimplicity.

Finally, item (3) follows from Proposition \ref{prop:linksHighlyConnected}.
\end{proof}

Denote by $S^n$ the (topological) sphere of dimension $n$.
A simplicial complex $K$ is said to be spherical if it has dimension $n\geq -1$ and its homotopy groups of degree $\leq n-1$ vanish.
Equivalently, $K$ is homotopy equivalent to a wedge of spheres of dimension $n$ (see Theorem \ref{thm:spherical}).

\begin{corollary}
\label{coro:spherical}
Let $K$ be a simplicial complex and $L$ a full subcomplex such that:
\begin{enumerate}
    \item $L$ is spherical of dimension $n-1$,
    \item $K\setminus L$ is a discrete set of vertices $V$ such that $\Lk_L(v)$ is spherical of dimension $n-1$ for all $v\in V$,
    \item for a fixed simplex $C\in L$, there are $(n-1)$-dimensional subcomplexes $(S_i)_{i\in I}$ such that $S_i\simeq S^{n-1}$ and $C\in S_i$ for all $i\in I$, and the homotopy classes $[S_i]$ generate the homotopy group $\pi_{n-1}(L,C)$, and
    \item for all $i\in I$, there exists $v\in V$ such that $S_i\leq \Lk_L(v)$.
\end{enumerate}
Then $K$ is spherical of dimension $n$.
\end{corollary}

\begin{proof}
By Proposition \ref{prop:linksHighlyConnected}, the map $i:L\hookrightarrow K$ is an $(n-1)$-equivalence since the links $\Lk_L(v)$, $v\in V$, are $(n-1)$-spherical by hypothesis.
Thus we have an epimorphism $\pi_{n-1}(L,C)\to \pi_{n-1}(K,C)$, where $C$ is as in item (3) of the hypotheses.

On the other hand, $\pi_{n-1}(L,C)$ is spanned by the classes of the subcomplexes $S_i$, which in fact lie in the links of the vertices $v\in V$.
Therefore, $[S_i]$ becomes the trivial element when regarded in $\pi_{n-1}(K,C)$, showing that $\pi_{n-1}(K,C)$ is the trivial group.
Hence $K$ is $(n-1)$-connected and $n$-dimensional, that is, $K$ is spherical of dimension $n$.
\end{proof}

We define now some particular simplicial complexes that look like those in the previous theorem.

\begin{definition}
    \label{def:complexExtended}
    Let $L$ be a simplicial complex and let $\F$ be a collection of groups acting on $L$.
    The \textit{$\F$-extended complex} $L_\F$ is the simplicial complex whose vertex set is the (disjoint) union of $\F$ and the vertex set of $L$, and simplices are as follows.
    Every simplex of $L$ is a simplex, and if $E\in \F$ fixes a simplex $\sigma\in L$ pointwise, then $\sigma \cup \{E\}$ is a simplex.
    That is, for $E\in \F$, $\Lk_L(E) = \Lk_{L_{\F}}(E) = L^E$ and $\St_{L_\F}(E) = L^E * \{E\}$.

    We say that $L$ and $\F$ have a \textit{compatible action by $G$} if $L$ is a $G$-complex, $\F$ is a $G$-set, and the action of $G$ on the vertices of $L_{\F}$ induces a simplicial action on this complex.    
\end{definition}

We will usually describe the homotopy type of $L_{\F}$ by using Theorem \ref{thm:MVsequence} since $L$ is a full subcomplex of $K = L_{\F}$ such that $K\setminus L = \F$ is a discrete set of vertices.

Recall that a simplicial action of a group $G$ on a simplicial complex $K$ is said to be admissible if every time an element $g\in G$ fixes a simplex $\sigma\in K$, then $g$ fixes $\sigma$ pointwise.

The following result is almost immediate from the Gluing Lemma \ref{lm:gluing} and the properties of a $G$-homotopy equivalence between $G$-complexes with admissible actions.

\begin{proposition}
\label{prop:extendingBySubgroups}
Let $L_1,L_2$ be two simplicial complexes with an admissible action of a group $G$ such that $L_1\simeq_G L_2$.
Let $\F$ be a family of subgroups of $G$.
Then ${L_1}_{\F}\simeq {L_2}_{\F}$.
Moreover, if $\F$ is closed under $G$-conjugation, then ${L_1}_{\F}\simeq_G {L_2}_{\F}$.
\end{proposition}

\begin{proof}
Let $\phi:|L_1|\to |L_2|$ be a $G$-homotopy equivalence.
For $H\leq G$, we have $|L_i|^H = |L_i^H|$ since the action is admissible.
Therefore, $\phi$ restricts to a homotopy equivalence $|L_1^H|\to |L_2^H|$ by the equivariant Whitehead Theorem \ref{thm:whitehead}.
Let $K_i = {L_i}_{\F}$ and $E\in \F$.
Then $\Lk_{L_i}(E) = L_i^E$.
By Theorem \ref{thm:MVsequence} and the Gluing Lemma \ref{lm:gluing}, we have a homotopy equivalence of the pushouts $\overline{\phi}:|K_1|\to |K_2|$.

Moreover, we claim that this is a $G$-homotopy equivalence.
Indeed, we must prove that $\overline{\phi}$ is $G$-equivariant and that it induces homotopy equivalences between the fixed point subspaces.

The $G$-equivariant property follows from $\phi$ and since $L_i,\F$ have a compatible $G$-action.
We analyse the fixed points.
Let $H\leq G$.
Then ${K_i}^H = (L_i^H)_{\F^H}$, where $\F^H = \{E\in\F\tq H \text{ normalises } E\}$.
For $E\in \F^H$, $\Lk_{L_i^H}(E) = L_i^{\gen{H,E}}$, and $\phi$ restricts again to a homotopy equivalence $L_1^{\gen{H,E}}\to L_2^{\gen{H,E}}$.
Thus, by the Gluing Lemma again, the induced map $\overline{\phi}_H : |K_1|^H \to |K_2|^H$ is a homotopy equivalence.
\end{proof}

\section{Homotopy equivalences between \texorpdfstring{$p$}{p}-subgroup complexes}
\label{sec:psubgroups}

We show some homotopy equivalences between different posets of $p$-subgroups whose order relations are not always given by set-theoretic inclusion, and related complexes of $p$-subgroups constructed by using Definition \ref{def:complexExtended}.

Recall that if $X$ is a poset, then the order complex of $X$, denoted by $\Delta(X)$, is the simplicial complex whose simplices are the finite totally ordered subsets of $X$.
We will regard a poset as a topological space via the topology of its order complex.
When we need to emphasise that we work on a simplicial complex, we will write expressions like $(\Delta(X))_{\F} \simeq Y$ to say that the $\F$-extended complex of $\Delta(X)$ is homotopy equivalent to the poset $Y$. Otherwise, we will just write $X\simeq Y$.
For $x\in X$, let $X_{\leq x} = \{y\in X\tq y\leq x\}$, and define analogously $X_{<x},X_{\geq x},X_{>x}$.

If $X$ is a $G$-poset, i.e., a poset with an action of a group $G$ by poset automorphisms, then $\Delta(X)$ is an admissible $G$-complex and $(\Delta(X))^H = \Delta(X^H)$ for any subgroup $H$ of $G$, where $X^H$ is the set of $H$-fixed points of $X$.
Also recall that if $f,g:X\to Y$ are two order-preserving $G$-maps between $G$-posets $X$ and $Y$ such that $f(x)\leq g(x)$ for all $x\in X$, then (the geometric realisations of) $f,g$ are $G$-homotopy equivalent.

From now on, $p$ denotes a prime number, and groups are assumed to be finite.
If $H\leq G$ are groups, then $N_G(H)$ and $C_G(H)$ denote the normaliser and centraliser of $H$ in $G$ respectively.
If $H,K\leq G$ then we write $N_G(H,K) = N_G(H)\cap N_G(K)$, and $[H,K]$ for the commutator subgroup of $H$ and $K$.
Also $Z(G)$, $O_p(G)$, $O^{p'}(G)$ denote respectively the centre of $G$, the largest normal $p$-subgroup of $G$ (the $p$-core), and the subgroup generated by the Sylow $p$-subgroups of $G$ (i.e., the smallest normal subgroup of index prime to $p$).
After fixing the prime $p$, $\Omega_1(G)$ is the subgroup generated by the order-$p$ elements of $G$.

The Quillen poset of $G$ (at $p$) is the poset $\Ap(G)$ of non-trivial elementary abelian $p$-subgroups of $G$ ordered by inclusion $\leq$.
As mentioned in the introduction, this poset was first studied by Quillen \cite{Qui78} who established numerous connections between its homotopy properties and algebraic properties of the group $G$.
Indeed, from Quillen's work it follows that $\Ap(G)$ is homotopy equivalent to the poset $\Sp(G)$ of all non-trivial $p$-subgroups of $G$.
There is another related $p$-subgroup poset, namely the Bouc poset $\Bp(G)$ whose elements are the radical $p$-subgroups of $G$, that is, non-trivial $p$-subgroups $R$ such that $O_p(N_G(R)) = R$ (ordered by inclusion).
This poset was introduced by S. Bouc \cite{Bouc} who showed that $\Bp(G) \simeq \Sp(G)$.

We will use the following slightly stronger result on the homotopy equivalences between the $p$-subgroup posets defined above.

\begin{proposition}
\label{prop:ApSpBpGHomotopyEquivalences}
Suppose we have groups $G,\hat{G}$, where $\hat{G}$ acts on $G$.
Then the inclusions $\Ap(G) \hookrightarrow \Sp(G)$ and $\Bp(G)\hookrightarrow \Sp(G)$ are $\hat{G}$-homotopy equivalences.
\end{proposition}

\begin{proof}
It is clear that $\hat{G}$ induces an action on $\Sp(G)$ and hence on $\Ap(G)$ and $\Bp(G)$.
We first apply Quillen's fibre Theorem \ref{thm:quillen} to the inclusion $i:\Ap(G)^{\op}\hookrightarrow \Sp(G)^{\op}$ (where $X^{\op}$ is the opposite poset).
Observe that for $P\in\Sp(G)\setminus \Ap(G)$ one has
\[ i^{-1}\big((\Sp(G)^{\op})_{\leq^{\op} P}\big) * (\Sp(G)^{\op})_{>^{\op}P} = \Sp(G)_{<P}.\]
If $\Phi(P)$ denotes the Frattini subgroup of $P$, then 
$\Sp(G)_{<P}$ is $\Stab_{\hat{G}}(P)$-contractible via the $\Stab_{\hat{G}}(P)$-equivariant chain of homotopies $R \leq R\Phi(P) \geq \Phi(P)$, for $R\in \Sp(G)_{<P}$, where $\Phi(P),R\Phi(P) \in \Sp(G)_{<P}$ since $P$ is not elementary abelian.
Hence, Theorem \ref{thm:quillen} yields the desired conclusion.

A similar argument works for $\Bp(G)$ since for $P\in \Sp(G)\setminus \Bp(G)$ we have that $\Sp(G)_{>P}$ is $\Stab_{\hat{G}}(P)$-contractible via the $\Stab_{\hat{G}}(P)$-equivariant chain of homotopies
\begin{equation}
    \label{eq:homotopyNonRadical}
    R \geq N_R(P) \leq N_R(P) O_p(N_G(P)) \geq O_p(N_G(P)),
\end{equation}
where each subgroup in this chain of inequalities lies in $\Sp(G)_{>P}$.
\end{proof}

Recall also the classical contractibility result due to Quillen, which we slightly generalise here:

\begin{proposition}
[Quillen]
\label{prop:contractibleAp}
If $O_p(G)\neq 1$ and $\hat{G}$ acts on $G$ then the posets $\Ap(G)$, $\Sp(G)$ and $\Bp(G)$ are $\hat{G}$-contractible.
\end{proposition}

\begin{proof}
We have a $\hat{G}$-equivariant homotopy
\[ P \leq PO_p(G)\geq O_p(G), \]
where each term lies in $\Sp(G)$.
Thus $\Sp(G)$ is $\hat{G}$-contractible, and the same holds for $\Ap(G)$ and $\Bp(G)$ by Proposition \ref{prop:ApSpBpGHomotopyEquivalences}.
\end{proof}

Next, we introduce a distinguished subposet $\F$ that will be crucial to recovering (somehow) the homotopy type of $\Ap(G)$ from  $\Ap(H)$, with $H\leq G$, and the behaviour of the Quillen poset of the centralisers in $H$ of the subgroups in $\F$.

\begin{definition}
[The $\F$-poset]
Let $H\leq G$ be finite groups, and fix a prime $p$.
Then we define
\[ \F_G(H)  = \{E\in \Ap(G)\tq E\cap H = 1\},\]
\[ \F_G(H)' = \{E\in \F_G(H)\tq O_p(C_H(E)) = 1\}.\]
\end{definition}

The following theorem is extracted from \cite{PS22}.

\begin{theorem}
\label{thm:equivalentPosetApG}
Let $H\leq K\leq G$ be finite groups and $p$ a prime.
Then $\Ap(K)$ is $N_G(H,K)$-homotopy equivalent to the poset
\[ \Bp(H\uparrow K) := \Bp(H) \cup \F_K(H),\]
whose ordering $\preceq$ is given as follows: inside $\Bp(H)$ and $\F_K(H)$ we keep the inclusion ordering, and for $R\in \Bp(H)$ and $E\in\F_K(H)$ we put
\[ E\prec R \quad \Leftrightarrow \quad C_R(E)\neq 1.\]
\end{theorem}

\begin{proof}
This is the poset $W^{\B}_K(H,1)$ of \cite[Theorem 4.10]{PS22}, where it is proved to be $N_K(H)$-homotopy equivalent to $\Ap(K)$.
Nevertheless, the same proof extends to an $N_G(H,K)$-homotopy equivalence.
\end{proof}

\begin{remark}
\label{rk:fixedPointsByPgroup}
Let $E$ be a $p$-group acting on a group $H$.
Then $\Sp(H)^E \simeq \Sp(C_H(E))$ (and hence we have the same homotopy equivalences for $\Bp$ and $\Ap$).
Indeed, $\Sp(C_H(E))\subseteq \Sp(H)^E$, and if $R\in \Sp(H)^E$, then $C_R(E)\neq 1$ by $p$-group actions, so the map $R \mapsto C_R(E)$ defines a strong deformation retract from $\Sp(H)^E$ onto $\Sp(C_H(E))$.
Note that $\Bp(H)^E \hookrightarrow \Sp(H)^E$ is a homotopy equivalence since the inclusion $\Bp(H)\hookrightarrow \Sp(H)$ is an $E$-equivariant homotopy equivalence by Proposition \ref{prop:ApSpBpGHomotopyEquivalences}.
Thus, we get
\begin{equation}
\label{eq:homotopyEquivalencesFPandBp}
\Bp(H)^E \simeq \Sp(H)^E \simeq \Sp(C_H(E)) \simeq \Bp(C_H(E)).
\end{equation}
Moreover, if $H,E\leq G$ and $E$ acts on $H$ by conjugation, then the equivalences in Eq. (\ref{eq:homotopyEquivalencesFPandBp}) are $N_G(H,E)$-equivariant.
\end{remark}

We will usually work under the assumption that $\F_G(H)$ consists of at most cyclic subgroups.
This is equivalent to saying that $\F_G(H)$ is an antichain, where we regard $\F_G(H)$ as a subposet of $\Ap(G)$ with order given by inclusion.

\begin{corollary}
\label{coro:ApAndExtended}
Let $H\leq K\leq G$ be such that $\F_K(H)$ is an antichain.
Then $\Ap(K)\simeq \Bp(H\uparrow K)$ is $N_G(H,K)$-homotopy equivalent to $(\Delta \Bp(H))_{\F_K(H)}$.
\end{corollary}

\begin{proof}
Let $\F := \F_K(H)$.
For $E\in \F$, note that $\Sp(C_H(E)) \hookrightarrow \Sp(H)^E$ is an $N_G(H,K,E)$-homotopy equivalence since the map $R\in \Sp(H)^E \mapsto C_R(E) \in \Sp(C_H(E))$ is an $N_G(H,K,E)$-equivariant homotopy inverse.

Let $\N^\S_H(E) = \{ R\in\Sp(H) \tq N_R(E)\neq 1\}$ and $\N^{\B}_H(E) = \Bp(H)\cap \N^{\S}_H(E)$.
Thus $\N^\B_H(E) \hookrightarrow \N^\S_H(E)$ is an $N_G(H,K,E)$-homotopy equivalence by Quillen's fibre Theorem \ref{thm:quillen}.
Indeed, if $R\in \N^\S_H(E)$ is not a radical $p$-subgroup of $H$, then, as $\N^\S_H(E)$ is an upward-closed subposet of $\Sp(H)$, we see that $\N^\S_H(E)_{>R} = \Sp(H)_{>R}$, which is $N_G(H,R)$-contractible (see Eq. (\ref{eq:homotopyNonRadical})).

We claim now that the inclusion $\Bp(H)^E\hookrightarrow \N^\B_H(E)$ also yields an $N_G(H,K,E)$-homotopy equivalence.
In fact, the map $c:\N^\S_H(E) \to \Sp(C_H(E))$ that takes $R$ to $ C_R(E)$ defines an equivariant strong deformation retract.
From this, we conclude that the inclusion $\Sp(H)^E\hookrightarrow \N^\S_H(E)$ is an $N_G(H,K,E)$-homotopy equivalence as well (see Remark \ref{rk:fixedPointsByPgroup}).

Then we have the following commutative diagram with $N_G(H,K,E)$-homotopy equivalences:
\[\xymatrix{
\Bp(H)^E \ar@{^(->}[d]_{\simeq} \ar@{^(->}[r] & \N^\B_H(E) \ar@{^(->}[d]_{\simeq} & \\
\Sp(H)^E \ar@{^(->}[r]_{\simeq} & \N^\S_H(E) \ar[r]^{c}_{\simeq} & \Sp(C_H(E))
}\]
Therefore $\Bp(H)^E\hookrightarrow \N^\B_H(E)$ is an $N_G(H,K,E)$-homotopy equivalence.

On the other hand, by Theorem \ref{thm:MVsequence} we have an adjunction
\[ \xymatrix{
\coprod_{E\in \F} |\Delta\Bp(H)^E| \ar@{^(->}[d] \ar[r] & |\Delta\Bp(H)| \ar@{^(->}[d] \\
\coprod_{E\in \F} |(\Delta\Bp(H)^E) * \{E\}| \ar[r]& | (\Delta\Bp(H))_{\F}|
}\]
Similarly, $|\Delta\Bp(H\uparrow K)|$ is obtained as the following adjunction space:
\[ \xymatrix{
\coprod_{E\in \F} |\Delta\N^\B_H(E)| \ar@{^(->}[d] \ar[r] & |\Delta\Bp(H)| \ar@{^(->}[d]\\
\coprod_{E\in \F} |(\Delta\N^\B_H(E)) * \{E\}| \ar[r] & |\Delta\Bp(H\uparrow K)|
}\]
Note that we have an $N_G(H,K)$-equivariant simplicial map $\Psi:{(\Delta\Bp(H))}_{\F} \to \Delta\Bp(H\uparrow K)$ which
is the identity on $\Bp(H)$ and on $\F$, and it maps a simplex $ \{R_1 < \ldots < R_k\} \cup \{E\}$ with $R_i\in \Bp(H)^E$ and $E\in \F$ 
to $\{E\prec R_1 < \ldots < R_k\}$.
Since $\Psi$ restricts to the homotopy equivalences described above, by the Gluing Lemma \ref{lm:gluing} we conclude that $\Psi$ is a homotopy equivalence.

Finally, the claim on the $N_G(H,K)$-homotopy equivalence follows by applying the equivariant Whitehead theorem.
Namely, for each $N\leq N_G(H,K)$ we take the fixed point subcomplexes by $N$ and observe that
\[ \big( (\Delta\Bp(H))_{\F}\big)^N = \big(\Delta \Bp(H)^N\big)_{\F^N}\]
and
\[ \big(\Delta \Bp(H\uparrow K) \big)^N = \Delta\big( \Bp(H)^N  \cup \F^N \big).\]
Then the same proof restricted to the $N$-fixed points yields a homotopy equivalence of the pushouts
\[\big((\Delta\Bp(H))_{\F}\big)^N \to  \big(\Delta \Bp(H\uparrow K) \big) ^N,\]
where this is the map induced by restricting the simplicial map $\Psi$ above.
\end{proof}

We produce one last homotopy equivalence between $p$-subgroup complexes.

\begin{proposition}
\label{prop:extendingByNormalSubgroups}
Let $H \leq K\leq G$ be such that $H,K$ are normal in $G$, and $\F_K(H) \cup \F_G(K)$ is an antichain.
Then there are $G$-homotopy equivalences:
\[ \Ap(G) \simeq_G \big(\, (\Delta \Bp(H))_{\F_K(H)'} \,\big)_{\F_G(K)'} \, \simeq_G \, \big(\, \Delta \Bp(K) \,\big)_{\F_G(K)'}. \]
\end{proposition}

\begin{proof}
First, we have
$$\Ap(G) \simeq_G (\Delta\Bp(K))_{\F_G(K)}$$
by Corollary \ref{coro:ApAndExtended}.
Also
\[ ( \Delta \Bp(H) )_{\F_K(H)} \simeq_G \Delta \Ap(K) \simeq_G \Delta\Bp(K)\]
by Corollary \ref{coro:ApAndExtended} and Proposition \ref{prop:ApSpBpGHomotopyEquivalences}.

Now, an element $E\in \F_K(H)\setminus \F_K(H)'$ satisfies
\[ \Lk_{\Delta\Bp(H)}(E) = \Delta(\Bp(H)^E) \simeq_{N_G(E)} \Bp(C_H(E))\simeq_{N_G(E)} *\]
by Remark \ref{rk:fixedPointsByPgroup} and Proposition \ref{prop:contractibleAp}.
Since the sets $\F_K(H)'$ and $\F_K(H)$ are $G$-invariant, by Theorem \ref{thm:MVsequence}(3) we get $G$-homotopy equivalences
\[ (\Delta \Bp(H))_{\F_K(H)'} \simeq_G ( \Delta \Bp(H) )_{\F_K(H)} \simeq_G \Delta \Bp(K).\]
Since $\F_G(K)$ is also closed under $G$-conjugation, the $G$-homotopy equivalence
\[  \big(\, (\Delta \Bp(H))_{\F_K(H)'} \,\big)_{\F_G(K)} \, \simeq_G \, \big(\, \Delta \Bp(K) \,\big)_{\F_G(K)} \]
follows from Proposition \ref{prop:extendingBySubgroups}.
Finally, the same argument as before now shows that we can equivariantly eliminate the vertices in $\F_G(K)\setminus \F_G(K)'$ and thus
\[  \big(\, (\Delta \Bp(H))_{\F_K(H)'} \,\big)_{\F_G(K)'} \, \simeq_G \, \big(\, \Delta \Bp(K) \,\big)_{\F_G(K)'}.\]
\end{proof}

One could even try to iterate the above procedure and consider longer chains of normal subgroups.
Moreover, all these results do not depend on the structure of the $\Bp$ poset, so they are still valid if we change $\Bp$ by $\Ap$ or $\Sp$.
However, in these notes, we will mainly consider the Bouc poset.

\section{Some preliminary results on buildings}
\label{sec:buildings}

In this section, we work with (spherical) buildings, as defined by Tits \cite{Tits}.
We adopt the language of \cite{AB} when we talk about buildings and groups with BN-pairs.
We will just review some basic facts, and we refer to \cite{AB,Tits} for standard definitions and further details.

Below, we recall the well-known Solomon-Tits theorem, which will be a key ingredient in the proof of our main theorem.

\begin{theorem}
[{Solomon-Tits}]
\label{thm:solomonTits}
Let $\Delta$ be a spherical building of rank $n$, and let $C$ be a fixed chamber.
Then $\Delta$ is homotopy equivalent to a wedge of spheres of dimension $n-1$.
Moreover, there is one sphere for each apartment of $\Delta$ containing $C$, and they yield a basis for the $(n-1)$-th homotopy group of $\Delta$ (and hence for the $(n-1)$-th homology group).
\end{theorem}

\begin{proof}
See Theorem 4.73 of \cite{AB}.
\end{proof}

Recall that a simplicial complex $K$ of finite dimension $n$ is said to be Cohen-Macaulay if for any simplex $\sigma\in K$ (including the empty simplex), the link $\Lk_K(\sigma)$ is spherical of dimension $n-|\sigma|$, where $|\sigma|$ is the size of $\sigma$.
Since in a building the links of simplices are buildings, we conclude that they are Cohen-Macaulay.

\begin{corollary}
\label{coro:CMBuildings}
Let $\Delta$ be a spherical building of rank $n$ with apartment system $\A$, and let $\sigma \in \Delta$ be a simplex.
Then $\Lk_{\Delta}(\sigma)$ is a building of rank $n-|\sigma|$ with apartment system given by $\A_{\Delta} = \{ \Lk_A(\sigma) \tq A \in \A \text{ and } \sigma \in A\}$.

In particular, $\Delta$ and $\Lk_{\Delta}(\sigma)$ are Cohen-Macaulay.
\end{corollary}

\begin{proof}
This follows from Proposition 4.9 of \cite{AB} together with its proof, and Theorem \ref{thm:solomonTits}.
\end{proof}

In conjunction with Corollary \ref{coro:spherical} we get:

\begin{corollary}
\label{coro:sphericalExtendedBuilding}
Let $\Delta$ be a building of rank $n$ and let
$K$ be a simplicial complex such that $\Delta \leq K$ is a full subcomplex and
$K\setminus \Delta = V$ is a discrete set of vertices.
Assume that $\Lk_{\Delta}(v)$ is spherical of dimension $n-1$ for all $v\in V$, and for some chamber $C\in \Delta$, any apartment of $\Delta$ containing $C$ is contained in some $\Lk_{\Delta}(v)$, $v\in V$.
Then $K$ is spherical of dimension $n$.

If in addition $\Lk_{\Delta}(v)$ is Cohen-Macaulay for all $v\in V$ and, for a given system of apartments $\A$, every apartment of $\A$ is contained in the link of some $v\in V$, then $K$ is Cohen-Macaulay.
\end{corollary}

\begin{proof}
That $K$ is spherical follows from the Solomon-Tits theorem and Corollary \ref{coro:spherical} applied to $L = \Delta$ where we use the classes of apartments containing $C$ as generating spheres $S_i$ in item (3) there.
Note that $K = \Delta_{V}$.

We show now that $K$ is Cohen-Macaulay under the additional assumptions.
Let $\A$ be a system of apartments for $\Delta$ such that for all $A \in \A$ there is $v\in V$ with $A \subseteq \Lk_K(v)$. Let $\sigma \in K$ and $v\in V$.
If $\sigma\in \Lk_K(v)$, then $\Lk_K(\sigma\cup \{v\}) = \Lk_{\Lk_{\Delta}(v)}(\sigma)$ is spherical of dimension $n -1 - |\sigma|$ by hypothesis.

On the other hand, $\Lk_K(\sigma) = \Lk_{\Delta}(\sigma)_{\F_{\sigma}}$ where $\F_{\sigma} = \{v'\in V\tq v'\in \Lk_K(\sigma)\}$.
Now, $\Lk_{\Delta}(\sigma)$ is a building by Corollary \ref{coro:CMBuildings}, and a system of apartments for it is given by $\Lk_A(\sigma)$ such that $A\in \A$ and $\sigma\in A$.
For such an apartment $A$, there is $v\in V$ with $A\subseteq \Lk_K(v)$ by hypothesis.
Hence $\sigma \in \Lk_K(v)$, $\Lk_A(\sigma) \subseteq\Lk_K(v)$, and therefore the apartment $\Lk_A(\sigma)$ of $\Lk_{\Delta}(\sigma)$ is contained in $\Lk_{\Lk_{\Delta}(\sigma)}(v)$.
By the first part of this corollary applied to the building $\Lk_{\Delta}(\sigma)$ inside the simplicial complex $\Lk_K(\sigma)$, we conclude that the latter is spherical (of the correct dimension).
\end{proof}

\section{Algebraic groups and main results}
\label{sec:algebraicgroups}

Now we apply the homotopical results described in the previous sections to the case $L = \Delta$ is the building of a finite group of Lie type $H$ and a vertex $v\in K\setminus \Delta$ can be regarded as a field automorphism of $H$ such that $\Lk_{\Delta}(v) = \Delta^v$ (the fixed point subcomplex) contains an apartment of $\Delta$ and it is homotopy equivalent to the building associated with $C_H(v)$.
Below we provide the necessary notation and basic results on linear algebraic groups and groups of Lie type.
Our main reference for algebraic groups is \cite{MalleTesterman}, but sometimes we will refer to \cite{GLS98, Steinberg, St2} for some general facts.
From now on, we denote by $\Delta(G)$ the building associated with a (not necessarily finite) group $G$ with a BN-pair.

Assume that $\KK$ is an algebraically closed field of characteristic $p>0$, and let $\Frob$ denote the Frobenius map which acts as $p$-powering on the entries of the matrices of $\GL_n(\KK)$.
For a (closed) linear algebraic group $\GG \leq \GL_n(\KK)$, we also write $\Frob$ for the induced Frobenius map on $\GG$.
A Steinberg endomorphism of $\GG$ is an algebraic group homomorphism $F:\GG\to \GG$ such that $F^m = \Frob^a$ for some $m,a\geq 1$.
Denote by $\Aut_0^+(\GG)$ the set of Steinberg endomorphisms of $\GG$.
The family of groups of Lie type in characteristic $p$, which we denote by $\Lie(p)$, consists of the finite groups $O^{p'}(\GG^F)$, where $\GG^F$ is the fixed point group of a Steinberg endomorphism $F$ on a simple algebraic group $\GG$ over an algebraically closed field of characteristic $p$:


\begin{center}
    $\Lie(p) = \{ O^{p'}(\GG^F) \tq \GG$ is a simple algebraic group in char $p$ and $F \in \Aut_0^+(\GG)\}$.
\end{center}


By a \textit{split-pair} of $\GG$ we mean a pair $(B,T)$ where $B$ is a Borel subgroup and $T\leq B$ is a maximal torus.
If $\GG$ is connected reductive with a maximal torus $T$, we fix a root system $\Sigma$ and base $\Pi$, with positive root system $\Sigma^+$, and therefore we have a Borel subgroup $B$ spanned by $T$ and the root subgroups $U_{\alpha}$, $\alpha\in \Sigma^+$ (cf. Theorem 11.1 of \cite{MalleTesterman}).
We call the tuple $(B,T,\Sigma,\Pi)$ a \textit{root-setup} for $\GG$, and $(B,T)$ is its associated split-pair.
When $\Sigma$ is indecomposable (e.g., $\GG$ is simple), we write $\Sigma = A_n,B_n$,\ldots to denote its type.


Suppose $\GG$ is a simple simply connected algebraic group with a root-setup $(B,T,\Sigma,\Pi)$.
A Steinberg endomorphism $F$ of $\GG$ can be written, up to an inner-automorphism, as $\gamma \Frob^s$, where $s>0$ and $\gamma$
is a group-theoretic automorphism of $\GG$ that arises from a permutation of the set of simple roots $\Pi$ that induce a symmetry of the Coxeter diagram.
Although $\gamma$ is a bijective endomorphism of algebraic groups, it might not be invertible as an algebraic group map.
For instance, in the cases $\Sigma = B_2,F_4,G_2$ with characteristic $p=2,2,3$ respectively, $\gamma$ can arise from the permutation that interchanges the long and short roots, which is an order-$2$ symmetry of the Coxeter diagram. In this case, we will use the notation $\gamma = \psi$ (here following \cite{GLS98}), and it holds $\psi^2 = \Frob$.
See, for example, Theorem 22.5 of \cite{MalleTesterman}.
For an arbitrary simple algebraic group $\GG$, we obtain the same decomposition for a Steinberg endomorphism $F$ after passing through a simply connected group $\GG_{sc}$ with an isogeny $\GG_{sc}\to \GG$ (cf. Proposition 9.15 of \cite{MalleTesterman} and Theorem 1.15.7 of \cite{GLS98}).

Moreover, $\gamma$ and $\Frob$ commute, and $B,T$ are (by construction) stable by $\gamma$ and $\Frob$.
If $F = \gamma \Frob$, we will say that $F$ is a Steinberg endomorphism in \textit{standard form} (with respect to the fixed root-setup).
Recall that if $x$ is an inner automorphism of $\GG$, then $\GG^{xF}$ and $\GG^F$ are isomorphic (see Corollary 21.8 of \cite{MalleTesterman}).
Therefore, to study the groups $O^{p'}(\GG^F)$ we can (and we will) assume that $F$ is in standard form.
Finally, we say that $O^{p'}(\GG^F)$ is a \textit{twisted group} if $\gamma \neq 1$, and \textit{untwisted} otherwise.

Let $H :=  O^{p'}(\GG^F)$.
If $F = \gamma \Frob^s$ is in standard form as above, any power of $\Frob$ induces an automorphism on $\GG^F$ and on $H$, which we call the \textit{canonical field automorphisms} of $H$.
We have a group homomorphism $\ZZ \to \Aut(H)$ given by $n\mapsto (\Frob|_{H})^n$, and we denote its image by $\Phi_{H}$.
This is a cyclic group of order $ds$, where $d$ is the order of the symmetry induced from $\gamma$ if $\gamma\neq \psi$ in the cases $\Sigma = B_2,F_4,G_2$ with $p=2,2,3$ respectively.
In such cases, $\Phi_{H}$ has order $2s+1$.

A \textit{field automorphism} of $H$ is an element $x \in \Aut(H)$ which is $\Aut(H)$-conjugate to an element of $\Phi_{H}$
(see also Section 10 of \cite{Steinberg}).
Note that this terminology differs from \cite{GLS98}, where a field automorphism for $H$ is a field automorphism in our sense but in addition, in the twisted cases for $\gamma\neq \psi$, its order must be prime to the order $d$ of the symmetry induced by $\gamma$ on the Coxeter diagram.
In such cases, if the order of $x$ is divisible by $d$ then 
\cite[Definition 2.5.13(c)]{GLS98} calls $x$ a graph automorphism.
We will not adopt this terminology here.

\begin{remark}
\label{rk:BpGandBuilding}
Let $F$ be a Steinberg endomorphism of a connected reductive algebraic group $\GG$ in characteristic $p>0$, with $\GG^F$ a finite group and $H = O^{p'}(\GG^F)$.
It is well-known that $F$ stabilises a split-pair $(B,T)$ of $\GG$ (cf. Corollary 21.12 of \cite{MalleTesterman}), and for any such a pair, 
$B^F,N_{\GG}(T)^F$ yields a BN-pair for $\GG^F$ (Theorem 24.10 of \cite{MalleTesterman}).
The poset of parabolic subgroups of $\GG^F$ gives rise to the building of $\GG^F$ defined from this BN-pair. This also defines a building for $H$, which we shall denote by $\Delta(H)$, where the parabolic subgroups are the parabolic subgroups of $\GG^F$ intersected with $H$.
Moreover, such a poset of parabolic subgroups is anti-isomorphic to the poset of radical $p$-subgroups of $H$ via the map $P\mapsto O_p(P)$ with inverse $R\mapsto N_{H}(R)$.
Therefore $\Delta \Bp(H) \cong \Delta(H)$ (and indeed this isomorphism is $\Aut(H)$-invariant).

Let $S$ be the set of simple reflections generating the Weyl group $W = N_{\GG}(T)/T$.
Then each proper subset $I$ of $S$ defines a \textit{standard parabolic subgroup} $P_I$ of $\GG$ (see Proposition 12.2 of \cite{MalleTesterman}).
Moreover, the order complex of the poset $\{ w^{-1} P_I w \tq w\in W, \ I\subsetneq S\}$ yields a \textit{standard apartment} for $\Delta(\GG)$.

Now, $F$ permutes the set $S$, and each $F$-invariant subset $I$ of $S$ defines an $F$-stable parabolic $P_I$, and $P_I^F$ is a \textit{standard parabolic subgroup} of $\GG^F$ (and indeed it is a parabolic subgroup for the BN-pair $B^F,N_{\GG}(T)^F$).
Hence, the order complex of the subposet of parabolics
\begin{equation}
    \label{eq:standardParabolic}
    \{ w^{-1} ((P_I)^F\cap H) w \tq w\in W^F,\ I\subsetneq S \text{ is $F$-invariant} \}, 
\end{equation}
yields a \textit{standard apartment} for $\Delta(H)$ (see p.320 of \cite{AB} and Section 26.1 of \cite{MalleTesterman}).
\end{remark}

\begin{lemma}
\label{lm:fieldAutomorphismsAndFixedPointsBuilding}
Let $x$ be a field automorphism of a group of Lie type $H\in \Lie(p)$.
Then the fixed point subcomplex $\Delta(H)^x$ contains an apartment that is $\Aut(H)$-conjugate to a standard apartment of $\Delta(H)$.
Moreover, if $|x|=p$ then $\Delta(H)^{x}$ is homotopy equivalent to $\Delta\Bp( O^{p'}( C_H(x) ) )$, which is a building of the same type as $\Delta(H)$.
\end{lemma}

\begin{proof}
Suppose that $H = O^{p'}(\GG^F)$, where $F$ is a Steinberg endomorphism of a simple algebraic group $\GG$ over an algebraically closed field of characteristic $p>0$, and fix a root-setup $(B,T,\Sigma,\Pi)$ for $\GG$.
Without loss of generality, we may assume that $F$ is in standard form, so $F=\gamma \Frob^s$ for some $s> 0$.
Let $W = N_{\GG}(T)/T$ be the corresponding Weyl group and $S$ the set of simple reflections generating $W$.
Then $\gamma$ permutes $S$ and $\Frob$ centralises $S$ (and hence $W$).
Let $Z$ be the set of parabolic subgroups of $H$ defined in Eq. (\ref{eq:standardParabolic}), so $\Delta(Z)$ is a standard apartment for $\Delta(H)$.
Up to $\Aut(H)$-conjugation, we can assume $x = F'|_H:H\to H$, with $F' = \Frob^a$ for some $a>0$.
Since $\Frob$ acts trivially on $W$ and stabilises $B$, we see that $x$ fixes every vertex in $Z$.
Therefore $\Delta(Z) \leq \Delta(H)^{x}$.

Suppose now that $|x|=p$.
The homotopy equivalence $\Delta(H)^x \simeq \Delta\Bp(C_H(x))$ follows from Remarks \ref{rk:fixedPointsByPgroup} and \ref{rk:BpGandBuilding}.
Finally, by Proposition \ref{prop:GLSprops} we have that $O^{p'}(C_H(x))\in \Lie(p)$ and its building has the same type as the building of $H$, with $\Delta(O^{p'}(C_H(x))) \cong \Delta\Bp(O^{p'}(C_H(x))) \cong \Delta\Bp(C_H(x))$.
This proves the Moreover part.
\end{proof}

Recall that for $H\in \Lie(p)$ we have $\Aut(H) \leq \Aut(H/Z(H))$, and the latter is a split-extension of $\Inndiag(H)$ by $\Phi_H \Gamma_H$ (see Theorems 2.5.12 and 2.5.14 of \cite{GLS98}).
Here $\Inndiag(H)$ denotes the group of inner-diagonal automorphisms of $H$, which has order prime to $p$, as well as $Z(H)$.
As defined above, $\Phi_H$ is the group of field automorphisms induced by the Frobenius map (after fixing a suitable root-setup and a Steinberg endomorphism in standard form), and $\Gamma_H$ is the set of graph automorphisms induced from symmetries of the underlying Coxeter diagram of $H$.
We write 
\begin{equation}
\Out(H)^* := \Aut(H) / \Inndiag(H),
\end{equation}
which is a subgroup of $\Phi_H \Gamma_H$ (with equality when $H=O^{p'}(\GG^F)$ and $\GG$ is simple and either simply connected or adjoint), and contains $\Phi_H$ as a normal subgroup.
We shall denote by $E^*$ and $x^*$ the image in $\Out(H)^*$ of a group $E$ and element $x$, respectively, inducing automorphisms on $H$.
In particular, if $x\in \Aut(H)$ induces a field automorphism on $H$ then $x^* \in \Phi_H$.
Recall $\Out(H)^*$ is abelian except for $H = D_4(q)$.
For more details, see Theorem 2.5.14 of \cite{GLS98}.
We warn that $\Gamma_H$ does not always define a subgroup of $\Out(H)^*$, and $\Gamma_H = 1$ if $H$ is twisted.

\bigskip

If a group $E$ acts on a group $H$, we say that $x\in E$ \textit{induces an outer automorphism on $H$} if the image of $x$ under the map $E\to \Aut(H)$ does not lie in $\Inn(H)$, the group of inner automorphisms of $H$.
We say that $E$ induces outer automorphisms on $H$ if every non-trivial element $x\in E$ induces an outer automorphism on $H$.

Recall that a subgroup $H$ of $G$ is self-centralising if $C_G(H)\leq H$, i.e., $C_G(H) = Z(H)$.
This property implies that, if $E\leq G$ normalises $H$ and $E \cap H = 1$, then $E \groupiso EH/H \leq \Out(H)$.
If in addition $H$ is normal in $G$, the elements of $\F_G(H)$ induce outer automorphisms on $H$.
An important example of a normal self-centralising subgroup is the generalised Fitting subgroup of $G$, which we will denote by $F^*(G)$.
From the properties of the components of $G$ and the Fitting subgroup, one can show that if $H$ is self-centralising, normal and quasisimple (i.e., $H=[H,H]$ is perfect and $H/Z(H)$ is simple) then $H = F^*(G)$.

We will be interested in the case 
that $G$ contains a self-centralising normal subgroup $H$ such that
$H \in \Lie(p)$:

\begin{definition}
\label{def:scnlp}
Let $G$ be a finite group and fix a prime $p$.
An $\name_p$-subgroup of $G$ is a self-centralising normal subgroup which is isomorphic to a finite group of Lie type in characteristic $p$.
\end{definition}

Note that if $H$ is an $\name_p$-subgroup of $G$, then $H = F^*(G)$ if and only if $H$ is quasisimple (see Theorem 2.2.7 of \cite{GLS98}).
Conversely, a group $H\in \Lie(p)$ is quasisimple except for the cases described in Table \ref{tab:excludedCases}.
In those cases, we see that either $[H,H]$ is a non-abelian simple group or $H$ is solvable.

From now on, we will adopt the following notation.
We write $H = {}^d\Sigma(q)$ if $H$ arises as the $O^{p'}$-subgroup of the fixed points of a Steinberg endomorphism $F = \gamma \Frob^s$ of a simple algebraic group $\GG$ with root-setup $(B,T,\Sigma,\Pi)$ such that $\gamma$ induces a symmetry of the Coxeter diagram of order $d$, and $q = p^{s}$.
In the special case that $\gamma = \psi$ and $\Sigma = B_2,F_4,G_2$ with $p=2,2,3$ resp., we write $q = p^{2s+1}$.

\begin{table}[ht]
\def\arraystretch{1.2}
\begin{tabular}{|c|c|c|c|}
    \hline
    $H$ & $p$ & $[H,H]$ & $\Out(H)$  \\
    \hline
    $A_1(2) \cong \Sym_3$ & $2$ &  Solvable & $1$\\
    $A_1(3) $ & $3$ & Solvable & $\Outdiag(H)\groupiso C_2$ \\
    ${}^2A_2(2)$ & $2$ & Solvable & $\Outdiag(H):\Phi_H \groupiso \Sym_3$ \\
    ${}^2B_2(2)$ & $2$ & Solvable & $1$\\
    $\PSp_4(2)\cong \Sym_6$ & $2$ & $\Alt_6$, simple & $\Gamma_H\groupiso C_2$ \\
    $G_2(2)\cong \Aut(\PSU_3(3))$ & $2$ & $\PSU_3(3)$, simple & $1$ \\
    ${}^2F_4(2)$ & $2$ & Tits group, simple & $1$\\
    ${}^2G_2(3)\cong \Aut(\PSL_2(8))$ & $3$ & $\PSL_2(8)$, simple & $1$\\
    \hline
\end{tabular}
\caption{In all cases, $|H:[H,H]|=p$ and $H\in \Lie(p)$.}
\label{tab:excludedCases}
\end{table}

In the following example, we analyse the homotopy type of the Quillen poset of the almost simple groups appearing in Table \ref{tab:excludedCases}.
Recall that $n_p$ denotes the $p$-part of a non-zero integer $n$, i.e., the largest power of $p$ dividing $n$.

\begin{example}
\label{ex:titsGroup}
Let $p = 2$ and let $H = {}^2F_4(2)$.
Then $Z(H) = 1$, $\Aut(H) = H$, and if $H$ is an $\name_2$-subgroup of a group $G$, then $H=G$.
However, $H$ is not simple, but $[H,H]$, the Tits group, is simple of index $2$ in $H$ (see Table \ref{tab:excludedCases}).
Since $[H,H]$ contains all the involutions of $H$, it follows that $\A_2(H) = \A_2([H,H])$.
Therefore $\A_2([H,H])$ is $H$-homotopy equivalent to a wedge of $1$-spheres, and the number of such spheres is $|H|_2 = 2^{12}$.

More generally, if $H$ is one of the non-solvable groups of the first column of Table \ref{tab:excludedCases} and $p$ is the index of $[H,H]$ in $H$, then it is not hard to show that $\Ap([H,H]) \simeq_H \Ap(H)\simeq_H  \Delta(H)$.
Therefore $\Ap([H,H])$ is homotopy equivalent to a wedge of $|H|_p$ spheres of dimension equal to the Lie rank of $H$ minus one.
\end{example}

Next, we show that $\name_p$-subgroups are unique.

\begin{lemma}
\label{lm:uniqueLiep}
Let $G$ be a finite group and $p$ a prime.
Then $G$ contains at most one $\name_p$-subgroup.
\end{lemma}

\begin{proof}
By way of contradiction, suppose that $G$ contains two distinct $\name_p$-subgroups $H$ and $K$.

By the discussion above on quasisimplicity and the generalised Fitting subgroup, we can assume without loss of generality that $H$ is not quasisimple, i.e., $H$ is one of the groups given in the first column of Table \ref{tab:excludedCases}.
Note that $K/K\cap H \groupiso KH/H$ embeds into $\Out(H)$, and the condition $[H,K]=1$ would imply $K\leq Z(H)$ and $H\leq Z(K)$, that is, $H,K$ are abelian, which never holds.
Thus $[H,K]\neq 1$, and in consequence $H\cap K\neq 1$.
Moreover, since $\Out(K),\Out(H)$ are solvable, we see that if $H$ or $K$ is solvable then $G$ is, and hence both $H,K$ are solvable (that is, $H$ is solvable if and only if $K$ is).

Suppose that $[H,H]$ is simple (and hence $K$ is not solvable).
Then $H\cap K = H$ or $[H,H]$, and since $[H,H]\notin \Lie(p)$, we see that $[H,H]<K$.
So $1\neq K/[H,H] \leq \Out(H)$, which forces $H = \PSp_4(2)$ and $H\leq G\leq \Aut(H)$.
From this, it is not hard to show that $H = K$.
Therefore, both $H$ and $K$ are solvable as in the first column of Table \ref{tab:excludedCases}.

Now, $\Out(H) = 1$ leads to $G/H = 1$, so $G = H$ and $K\leq H$.
But a case-by-case inspection shows that $H$ does not contain a proper self-centralising normal subgroup $K\in \Lie(p)$.
Therefore $\Out(H) \neq 1$, which implies that $H,K$ lie in the same row of Table \ref{tab:excludedCases}.
Again, a case-by-case inspection with GAP shows that one must have $H = K$ (to see this, pass through the quotient $G/Z(H)$ and analyse the possibilities of $G$ and $K/(Z(H)\cap K)$ as a normal subgroup of $G/Z(H)$).
\end{proof}

In light of the previous lemma, we can talk about the ``unique" (if it exists) self-centralising normal subgroup $H\in \Lie(p)$ of a group $G$.
Note that $G$ might still contain a self-centralising normal subgroup $K\in\Lie(r)$ for a different prime $r$.
For example, $H = \PSp_4(2)\in \Lie(2)$ and $K = [H,H]\groupiso\Alt_6 \groupiso \PSL_2(9)\in \Lie(3)$ are both self-centralising normal subgroups of any $G$ such that $H \leq G\leq \Aut(H)$.

\bigskip

Next, we introduce a useful notation that arranges different types of order-$p$ outer automorphisms of a group of Lie type in characteristic $p$.

\begin{definition}
Let $G$ be a finite group with an $\name_p$-subgroup $H$.
Then we write
\begin{align*}
    \F_f & := \{ E\in \F_G(H) \tq |E|=p \text{ induces field automorphisms on }H\},\\
    \F_g & := \{ E\in \F_G(H) \tq |E|=p, \ E\notin \F_f \text{ and } O_p(C_{H}(E)) = 1\},\\
    \F_c & := \{ E\in \F_G(H) \tq |E|=p, \ O_p(C_{H}(E)) \neq  1\}.
\end{align*}
\noindent
Note that $\F_f$, $\F_g$ and $\F_c$ are (possibly empty) antichains, and they are uniquely determined by $p$ and $G$ by Lemma \ref{lm:uniqueLiep}.
\end{definition}

The elements of $\F_g$ are roughly the order-$p$ subgroups inducing \textit{graph} or \textit{graph-field} automorphisms on $H$ whose image in $\Out(H)^*$ do not lie in $\Phi_{H}$ (cf. Definition 2.5.13 of \cite{GLS98}).

Below we quote the crucial parts of Propositions 4.9.1 and 4.9.2 of \cite{GLS98} on centralisers of order-$p$ outer automorphisms of groups of Lie type in characteristic $p$.

\begin{proposition}
\label{prop:GLSprops}
Let $H\in \Lie(p)$ and let $x\in \Aut(H)\setminus \Inndiag(H)$ be an order-$p$ element.
The following hold:
\begin{enumerate}
    \setlength{\itemsep}{0.4em}
    \item (Field) If $x^* \in \Phi_H$ then exactly one of the following holds:
    \begin{enumerate}
        \item $d\neq p$, $x$ is a field automorphism of $H$ and $O^{p'}(C_H(x)) \groupiso {}^d\Sigma(q^{1/p}) \in \Lie(p)$.
        Further, if $Z(H) = 1$ then $Z(C_H(x)) =1$ and $C_{\Inndiag(H)}(x) = \Inndiag({}^d\Sigma(q^{1/p}))$.
        \item $d = p$ and $O_p(C_H(x)) \neq 1$.
        \item $d = p$, $x$ is a field automorphism of $H$ and $O^{p'}(C_H(x))\in \Lie(p)$ is described in Table \ref{tab:twistedCases}.
        \begin{table}[ht]
            \def\arraystretch{1.2}
            \centering
            \begin{tabular}{|c|c|c|}
            \hline
            $p$ & $H$ & $O^{p'}(C_H(x))$ \\
            \hline
            $2$ & ${}^2A_{n-1}(q)$ & $B_{[n/2]}(q)$\\
            $2$ & ${}^2D_{n}(q)$ & $B_{n-1}(q)$\\
            $2$ & ${}^2E_6(q)$ & $F_4(q)$\\
            $3$ & ${}^3D_4(q)$ & $G_2(q)$.\\
            \hline
            \end{tabular}
            \caption{Field automorphisms in the twisted cases.}
            \label{tab:twistedCases}
        \end{table}
    \end{enumerate}
    \noindent
    In cases (a) and (c), $H$ and $O^{p'}(C_H(x))$ have the same Lie rank, and if $y\in \Aut(H)\setminus \Inndiag(H)$ also has order $p$, $y^* \in \Phi_H$ and case (b) does not hold for $y$, then $\gen{x}$ and $\gen{y}$ are $\Inndiag(H)$-conjugate.
    
    \item (Non-field) If $x^* \in \Out(H)^* \setminus \Phi_H$ then $d = 1$ (so $H$ is untwisted), $x$ is not a field automorphism of $H$, $p=2$ or $3$, and either $O_p(C_H(x)) \neq 1$ or $O^{p'}(C_H(x))\in \Lie(p)$ is described in Table \ref{tab:decreasingRanks}. 

    Moreover, if $y\in \Aut(H)\setminus \Inndiag(H)\Phi_H$ has order $p$ and $y$ falls in the same row as $x$ in Table \ref{tab:decreasingRanks} then $\gen{x}$ and $\gen{y}$ are $\Inndiag(H)$-conjugate, except when $H = D_4(q)$ for $p=2$ in the graph case, or $p=3$ in the graph-field case.
    In those cases $\gen{x},\gen{y}$ are $\Inndiag(H)\Gamma_H$-conjugate, with $\Gamma_H \cong \Sym_3$.
    
    \begin{table}[ht]
    \centering
    \def\arraystretch{1.2}
    \begin{tabular}{|c|c|c|c|c|c|}
    \hline
    $p$ & $H$ & Rank of $H$ & Type of $x$ & $O^{p'}(C_{H}(x))$ & Rank of $O^{p'}(C_{H}(x))$\\
     \hline
   $2$ & $B_2(2^{2a+1})$ & $2$ & graph & ${}^2B_2(2^{2a+1})$ & $1$\\
   $2$ & $F_4(2^{2a+1})$ & $4$ & graph & ${}^2F_4(2^{2a+1})$ & $2$\\
   $2$ & $A_{n-1}(2^a)$, $n\geq 3$ & $n-1$ & graph & $B_{[n/2]}(2^a)$ & $[n/2]$\\
   $2$ & $A_{n-1}(2^{2a})$, $n\geq 3$ & $n-1$ & graph-field & ${}^2A_{n-1}(2^{a})$ & $[n/2]$  \\
   $2$ & $D_n(2^a)$, $n\geq 4$ & $n$ & graph & $B_{n-1}(2^a)$ & $n-1$ \\
   $2$ & $D_n(2^{2a})$, $n\geq 4$ & $n$ & graph-field & ${}^2D_n(2^{a})$ & $n-1$\\
   $3$ & $D_4(3^a)$ & $4$ & graph & $G_2(3^a)$ & $2$\\
   $3$ & $D_4(3^{3a})$ & $4$ & graph-field & ${}^3D_4(3^{a})$ & $2$\\
   $2$ & $E_6(2^a)$ & $6$ & graph & $F_4(2^a)$ & $4$\\
   $2$ & $E_6(2^{2a})$ & $6$ & graph-field & ${}^2E_6(2^{a})$ & $4$\\
    \hline
    \end{tabular}
    \caption{Centralisers for $x\in \Aut(H)$ of order $p$ such that $x^* \in \Out(H)^* \setminus \Phi_{H}$ and $O_p(C_H(x))=1$.}
    \label{tab:decreasingRanks}
    \end{table}
\end{enumerate}
\end{proposition}

We derive some useful properties from the previous proposition.

\begin{proposition}
\label{prop:FfFgFc}
Let $G$ be a finite group with an $\name_p$-subgroup $H$, and let $E\in \F_G(H)$.
\begin{enumerate}[label=(\arabic*)]
    \item $E\groupiso E^* \leq \Out(H)^*$ and it has $p$-rank at most $2$.
    \item If $H$ is untwisted, then $E\in \F_f$ if and only if $E$ is cyclic and $E^* \leq \Phi_H$.
    \item If $H$ is untwisted, then $E\in \F_g\cup \F_c$ if and only if $E$ is cyclic and $E^*\cap \Phi_H = 1$.
    Moreover, $\F_g$ and $\F_c$ are disjoint, and if $\F_c$ is non-empty then $\F_g$ is non-empty.
    \item If $H$ is twisted, then $\F_g = \emptyset$, $E\in \F_f\cup \F_c$, $E^*\leq \Phi_H$, and $\F_f$ and $\F_c$ are disjoint.
    \item If $E$ is cyclic then $E \in \F_f\cup \F_g\cup \F_c$, which is a disjoint union.
    \item If $E\in \F_c$ then $\Bp(K)^E \simeq *$ for any $H\leq K\leq G$.
    
    \item Let $\mathrm{\bf q}:G\to \Aut(H)$ be the map induced by conjugation on $H$, and define
    \[ G_d:=\mathrm{\bf q}^{-1}(\Inndiag(H)) \ \text{ and } \ G_{df} := \mathrm{\bf q}^{-1}(\Inndiag(H)\Phi_H).\]
    Then $\F_{G_{df}}(H) = \F_f$ if $H$ is untwisted, and $\F_{G_{df}}(H) = \F_f \cup \F_c$ otherwise.
    In particular, $\F_{G_{df}}(H)$ is an antichain of order-$p$ subgroups.
    

\end{enumerate}
\end{proposition}

\begin{proof}
If $E\in \F_G(H)$ then $E$ is faithful on $H$ and $E\cap \Inndiag(H) = 1$ since $\Outdiag(H)$ and $Z(H)$ have order prime to $p$.
Thus $E\groupiso E^*\leq \Out(H)^*$.

Items (2-5) follow directly from Proposition \ref{prop:GLSprops}.
We briefly explain that $\F_g$ is non-empty if $\F_c$ is non-empty and $H$ is untwisted.
Take $Z$ to be a long-root subgroup of $H$, $y\in \Aut(H)$ an order-$p$ element inducing ``graph" automorphisms on $H$ (according to the rows of Table \ref{tab:decreasingRanks}), and suppose $\gen{x}\in \F_c$.
Then there exists $\varphi\in \Aut(H)$ such that $\gen{\varphi(yz)} = \varphi(\gen{yz}) = \gen{x}$ for some $z\in Z$, by Proposition 4.9.2(g) of \cite{GLS98} and the remark below that proposition.
Since $\varphi(z)\in H$, we see that $\varphi(y)\in H\gen{x}\leq G$, so $\gen{\varphi(y)} \in \F_g$.

For item (6) recall first that $\Bp(K)^E\simeq \Bp(C_K(E))$ by Remark \ref{rk:fixedPointsByPgroup}.
Also since $H\leq K$ is normal, $C_{H}(E)$ is normal in $C_K(E)$, so $1\neq O_p(C_{H}(E)) \leq O_p(C_K(E))$.
This implies $\Bp(C_K(E))\simeq *$ by Proposition \ref{prop:contractibleAp}.

On the other hand, $\F_f\subseteq \F_{G_{df}}(H)$.
Reciprocally, if $E\in \F_{G_{df}}(H)$ then $E^* \leq \Phi_H$ and the latter is a cyclic group, so $E$ has order $p$.
By items (2) and (4), $E\in \F_f$ if $H$ is untwisted, and $E\in \F_f\cup \F_c$ otherwise.
This establishes (7).
\end{proof}

Now we have the necessary preliminaries to prove the first part of our main theorem.

\begin{theorem}
[{Field case}]
\label{thm:fieldCase}
Let $p$ be a prime and $G$ a finite group with an $\name_p$-subgroup $H$ of Lie rank $n$ such that $\F_g = \emptyset$.
Then
 \[ \Ap(G) \simeq_{G} (\Delta\Bp(H))_{\F_f} \cong_{G} \Delta(H)_{\F_f}.\]
Moreover, $\Delta(H)_{\F_f}$ is spherical of dimension $n$ if $\F_f\neq \emptyset$, and $n-1$ otherwise.
In any case,
\begin{align*}
    \widetilde{H}_{n}( \Ap(G),\ZZ ) & \groupiso_G \Ker\left(\, \bigoplus_{\overline{E}\in \F_f/G} \Ind_{N_G(E)}^G\big(\widetilde{H}_{n-1}( \Delta(H)^E ,\ZZ )\big) \overset{i'}{\longrightarrow} \widetilde{H}_{n-1}( \Delta(H),\ZZ) \,\right)\\
    & \groupiso_G \Ker\left(\, \bigoplus_{\overline{E}\in \F_f/G} \Ind_{N_G(E)}^G\big(\widetilde{H}_{n-1}( \Ap(C_H(E)) ,\ZZ )\big) \overset{i}{\longrightarrow} \widetilde{H}_{n-1}( \Ap(H),\ZZ) \,\right),
\end{align*}
where the maps $i',i$ are induced from the corresponding inclusions.
\end{theorem}

\begin{proof}
By Remark \ref{rk:BpGandBuilding}, the isomorphism of $\Bp(H)^{\op}$ with the poset of parabolic subgroups of $H$ is also a $G$-isomorphism since $H$ is characteristic in $G$ by uniqueness Lemma \ref{lm:uniqueLiep}.
Then we have a $G$-isomorphism $\Delta \Bp(H) \cong_{G} \Delta(H)$, which naturally extends to a $G$-isomorphism
\[ \Delta \Bp(H)_{\F_f} \cong_{G} \Delta(H)_{\F_f}.\]

On the other hand, by Proposition \ref{prop:FfFgFc} we conclude that $\F_G(H)' = \F_f$, and thus by Proposition \ref{prop:extendingByNormalSubgroups} applied to $H$ and $K=G$, we get a $G$-homotopy equivalence
\[ \Ap(G) \simeq_G (\Delta\Bp(H))_{\F_f}.\]

Next, we prove that this complex is spherical, and since $\Delta \Bp(H) \cong \Delta(H)$ is already spherical by the Solomon-Tits Theorem \ref{thm:solomonTits}, we can assume that we are in the case that $\F_f$ is non-empty.
For $E\in \F_{G}(H)$, we have that $\Lk_{\Delta(H)}(E) = \Delta(H)^E$, which is spherical of dimension $n-1$ and contains an apartment $A$ such that for a suitable $y\in \Aut(H)$, $A' = \{\sigma\cdot y \tq \sigma \in A\}$ is an apartment of the standard apartment system $\A'$ of $\Delta(H)$ by Lemma \ref{lm:fieldAutomorphismsAndFixedPointsBuilding} (here $\sigma\cdot y$ denotes the simplicial action on the right of $\Aut(H)$ on $\Delta(H)$).
After acting by an element of $\Aut(H)$, we get an apartment system $\A$ of $\Delta(H)$ that contains an apartment fixed by some $E\in \F_f$.
Now recall that $H$ is transitive on $\{ (C,A) \tq A\in \A, \text{ $C$ is a chamber in }A\}$ (see Theorem 6.56 of \cite{AB}).
In particular, if we take a chamber $C$ lying in an apartment $A_0\in \A$ fixed by some $E_0\in \F_f$, then $\Stab_H(C)$ is transitive on the set of apartments of $\A$ containing $C$, so each apartment of $\A$ containing $C$ is contained in $\Lk_{\Delta(H)}(E)$ for some $\Stab_H(C)$-conjugate $E\in \F_f$ of $E_0$.
By Corollary \ref{coro:sphericalExtendedBuilding} we conclude then that $\Delta(H)_{\F_f}$ is spherical of dimension $n$.

Finally, the isomorphism in homology in terms of $\Delta(H)^E$ and $\Delta(H)$ follows from Theorem \ref{thm:MVsequence}(1) applied to $K = \Delta(H)_{\F_f}$ and $L = \Delta(H)$.
For the case of $\Ap(C_H(E))$ and $\Ap(H)$, it also follows from Theorem \ref{thm:MVsequence}(1) applied to $K = \Delta \Ap(G)$ and $L = \Delta \N_G(H) \simeq_G \Ap(H)$, where $\N_G(H)  = \Ap(G)\setminus \Ap(H)$, so $K\setminus L = \F_f$.
\end{proof}


Let us view some examples and determine the number of spheres that appear in such a wedge decomposition.

\begin{example}
\label{ex:PSL2}
Let $p = 2$, $q$ a power of $2$, $H = \PSL_2(q^2)$ and $G = \Aut(H) = H:\Phi_H$.
Then $\F_f/G$ consists of a unique orbit represented by the unique order-$2$ subgroup of $\Phi_H$, which we can denote by $\Phi$.
Also $N_G(\Phi) = C_H(\Phi)\Phi_H = \PSL_2(q)\Phi_H$.
Since $\widetilde{\chi}(\A_2(H)) = |H|_2 = q^2$ and $\widetilde{\chi}(\A_2(C_H(\Phi))) = |C_H(\Phi)|_2=q$,
by Theorem \ref{thm:fieldCase} we see that $\A_2(G)$ is homotopy equivalent to a wedge of $1$-spheres and
\[ \dim \widetilde{H}_1(\A_2(G),\QQ) = \frac{|G|}{|N_G(\Phi)|}q - q^2 = \frac{q^2(q^4-1)}{q(q^2-1)}q - q^2 = q^4.\]
\end{example}

\begin{example}
\label{ex:2DnChar2}
Suppose $p=2$, $q$ is a power of $2$, $H = \POmega^-_{2n}(q)$ $(\, ={}^2D_n(q))$, and $G = \Omega_1(\Aut(H))$.
Then $G = H:\Phi$, where $\Phi$ is the unique subgroup of order $2$ of $\Phi_{H}$.
Recall that $H$ is obtained from a symmetry of order $2$, and thus $\Phi_H$ has order $2s$, where $q=p^s$.
Moreover, since $\Inndiag(H) = H$, we see that elements of $\F_f$ are all $G$-conjugate to $\Phi$ and
\[ N_G(\Phi) = C_G(\Phi) = \PSp_{2n-2}(q) \times \Phi\]
by Proposition \ref{prop:GLSprops}(1)(c).
Here we have used that $B_n(2^a) \groupiso C_n(2^a)$ (see \cite[Theorem 37]{Steinberg}).
Hence, by Theorem \ref{thm:fieldCase}, $\A_2(G)$ is homotopy equivalent to a wedge of $(n-1)$-spheres, and the number of such spheres is
\begin{equation}
\label{eq:dimensionField2DnChar2}
    \dim \widetilde{H}_{n-1}(\A_2(G),\QQ) = \frac{|G|}{|N_G(\Phi)|}\cdot q^{(n-1)^2} -q^{n(n-1)} = q^{n(n-1)}((q^n+1) - 1) = q^{n^2}.
\end{equation}
\end{example}

Furthermore, our previous theorem can be used to show that $\Bp(\Omega_1(G))$ is spherical of dimension $n$ if $\F_f\neq\emptyset$:

\begin{corollary}
\label{coro:sphericalBp}
Let $p$ be a prime and $G$ a finite group with an $\name_p$-subgroup of Lie rank $n$ such that $\F_g = \emptyset$ and $\F_f\neq\emptyset$.
Then $\Bp(\Omega_1(G))$ is spherical of dimension $n$.
\end{corollary}

\begin{proof}
Let $H$ be the $\name_p$-subgroup of $G$.
Since $\F_f\neq\emptyset$, $H$ is not one of the groups in Table \ref{tab:excludedCases}, and therefore $H$ is quasisimple with $\Omega_1(H) = H$. In particular, $H$ is an $\name_p$-subgroup of $\Omega_1(G)$ and $\F_G(H) = \F_{\Omega_1(G)}(H)$, so we may assume that $G = \Omega_1(G)$.
We claim that $|G|_p = p\cdot |H|_p$.
From $\F_g = \emptyset$, we also get $\F_c = \emptyset$ if $H$ is untwisted by Proposition \ref{prop:FfFgFc}(3), and therefore $G^* \leq \Omega_1(\Phi_H)$, which is a cyclic group of order $p$.
In any case, $G^* = \Omega_1(\Phi_H)$ by the assumption $\F_f\neq\emptyset$.
This proves our claim.

Now we invoke a crucial property on radical $p$-subgroups: If $N$ is a normal subgroup of $K$, and $R\in \Bp(K)$, then $N\cap R\in \Bp(N)\cup \{1\}$.
In our situation, if $R\in \Bp(G)$ then $R\cap H=1$ implies that $R$ is cyclic of order $p$ and hence $R\in \F_f$.
Moreover, if $R < T$ lie in $\Bp(G)$ and $R\cap H = T\cap H$ then
we claim that $R\leq H$ and $R$ has index $p$ in $T$.
Indeed, $|R:R\cap H| \leq p$, and $|T:T\cap H|\leq p$.
If it happens that $R \neq R\cap H$ then both indices are $p$, and hence $|R| = p |R\cap H| = p |T\cap H| = |T|$, i.e., $R = T$.
Thus $R\leq H$, and again since  $T > R = R\cap H = T\cap H$, we see that $R$ has index $p$ in $T$.

These computations show that $\Bp(G)$ has dimension at most one plus the dimension of $\Bp(H)$.
Since $\Bp(G)$ is homotopy equivalent to a non-trivial wedge of $n$-spheres by Theorem \ref{thm:fieldCase} (see also Corollary \ref{coro:eulerFieldCase}), we see that $\Bp(G)$ has dimension exactly $n$.
Hence $\Bp(G)$ is spherical.
\end{proof}

In general, $\Bp(G)$ might not have dimension $n$:

\begin{example}
Let $p=2$, $H = \PSL_2(16)$ and $G = \Aut(H) = H:\Phi_H$, where $\Phi_H$ is cyclic of order $4$.
Then $\B_2(G)$ is homotopy equivalent to a wedge of $256$ $1$-spheres by Example \ref{ex:PSL2}, but $\B_2(G)$ has dimension $2$: 
We have a chain of radical $2$-subgroups given by $\Omega_1(\Phi_H) < \Phi_H < S\Phi_H$, where $S$ is a Sylow $2$-subgroup of $H$ normalised by $\Phi_H$.
\end{example}

Theorem \ref{thm:fieldCase} describes the homotopy type of $\Ap(G)$ when $G$ contains at most field automorphisms of order $p$, that is, $\F_g = \emptyset$, which is indeed the case for most groups in $ \Lie(p)$.
The following theorem determines the homotopy type of $\Ap(G)$ in the case
$\F_f = \emptyset$, which includes then $\F_g\neq\emptyset$ (and hence $p=2,3$).

\begin{theorem}
[{No field case}]
\label{thm:noFieldCase}
Let $p$ be a prime and $G$ a finite group with an $\name_p$-subgroup $H$ of Lie rank $n$ such that $\F_f = \emptyset$.
Then $\Ap(G) \simeq_{G} \Delta \Bp(H)_{\F_g} \cong_{G} \Delta(H)_{\F_g}$, and the latter is homotopy equivalent to a wedge of spheres of dimension $n-1$ and dimension $m_E$, for each $E\in\F_g$ such that $O^{p'}(C_H(E))$ has rank $m_E$ (and no such spheres appear if $\F_g$ is empty).

Moreover, the inclusion $\Ap(H)\hookrightarrow \Ap(G)$ induces an injection in homology, and we have the following $G$-isomorphism of $G$-modules for all $m\geq 0$:
\[ \widetilde{H_m}(\Ap(G),\QQ) \cong \widetilde{H}_{m}(\Ap(H),\QQ) \oplus \bigoplus_{\overline{E}\in \F_g/G} \Ind_{N_G(E)}^G \big(\, \widetilde{H}_{m-1}(\Ap(C_H(E)),\QQ) \,\big).\]
\end{theorem}

\begin{proof}
Proposition \ref{prop:FfFgFc} and $\F_f=\emptyset$ imply that $\F_G(H)' = \F_g $ is an antichain (possibly empty).
Thus, similar to the proof of Theorem \ref{thm:fieldCase}, we have
$$\Ap(G) \simeq_{G} \Delta \Bp(H)_{\F_g} \cong_{G} \Delta(H)_{\F_g}.$$

Now, if $E\in \F_g$, then $\Lk_{\Delta(H)}(E) \cong \Delta(\Bp(H)^E) \simeq_{N_G(E)}  \Delta(\Bp(C_H(E)))$ by Remark \ref{rk:fixedPointsByPgroup}, and also
\[ \Bp(C_H(E)) = \Bp(O^{p'}(C_H(E))).\]
From Table \ref{tab:decreasingRanks}, we see that $O^{p'}(C_{H}(E)) \in \Lie(p)$ and $\Bp(O^{p'}(C_{H}(E)))$ has dimension $m_E-1$, which is strictly less than $n-1$.
Therefore, $\Lk_{\Delta(H)}(E)$ is $N_G(E)$-homotopy equivalent to a spherical complex of dimension strictly less than $n-1$, so the inclusion $\Lk_{\Delta(H)}(E) \hookrightarrow \Delta(H)$ is homotopy equivalent to a constant map by Lemma \ref{lm:nullHomotopicDimension}, for all $E\in \F_g$.
We get the desired conclusions by Theorem \ref{thm:MVsequence}(2).
\end{proof}

Now we can prove our main result:

\begin{theorem}
\label{thm:main}
Let $p$ be a prime and $G$ a finite group with an $\name_p$-subgroup $H$ of Lie rank $n$.
Then $\Ap(G)$ is homotopy equivalent to a wedge of spheres of dimension at most $n$.

Write $m_E$ for the Lie rank of $O^{p'}(C_H(E))$ if $E\in \F_g$.
Set $m_E^*=m_E-1$ if some element of $\F_f$ commutes with $E$, and $m_E^*=m_E$ otherwise.
The dimensions of the spheres that appear in the wedge decomposition of $\Ap(G)$ are as follows:
\begin{enumerate}
    \item If $\F_f = \emptyset$ then $(n-1)$-spheres appear.
    \item If $\F_f \neq \emptyset$ then $n$-spheres appear.
    \item If $E\in \F_g$ then spheres of dimension $m_E^*+1$ appear.
\end{enumerate}
Finally, if $\F_f\neq\emptyset$, we have the following $G$-module decomposition in homology.
Let $G_{df}$ be the group of Proposition \ref{prop:FfFgFc}.
Then
\begin{equation*}
\begin{aligned}
    \widetilde{H}_{*}( \Ap(G_{df}), & \ZZ )  = \widetilde{H}_{n}( \Ap(G_{df}),\ZZ )\\
    & = \Ker\left(\, \bigoplus_{\overline{E}\in \F_f/G} \Ind_{N_G(E)}^G\big(\widetilde{H}_{n-1}( \Ap(C_H(E)) ,\ZZ )\big) \overset{i}{\longrightarrow} \widetilde{H}_{n-1}( \Ap(H),\ZZ) \,\right),
\end{aligned}    
\end{equation*}
where $i$ is the $G$-equivariant surjective map induced by the inclusions $C_H(E)\leq H$,
and for $m\geq 0$ we have
\[ \widetilde{H}_m(\Ap(G),\QQ) = \widetilde{H}_m(\Ap(G_{df}),\QQ) \oplus \bigoplus_{\overline{E}\in \F_g/G} \Ind_{N_G(E)}^G\big( \widetilde{H}_{m-1}(\Ap(C_{G_{df}}(E),\QQ) \big),\]
where $\widetilde{H}_{m-1}(\Ap(C_{G_{df}}(E),\QQ) \big)\neq 0$ if and only if $m-1 = m_E^*$.
\end{theorem}

\begin{proof}
Let $G$ and $H$ be as in the statement.
By Theorems \ref{thm:fieldCase} and \ref{thm:noFieldCase}, we may assume $\F_f,\F_g\neq\emptyset$.

By Proposition \ref{prop:FfFgFc}, $G_{df}$ is normal in $G$ and $\F_{G_{df}}(H)'=\F_f$.
We also claim that $\F_G(G_{df})'=\F_g$.
Indeed, the containment $\F_g\subseteq \F_G(G_{df})$ is clear, and thus it remains to see that elements of $\F_G(G_{df})$ have order $p$.
But this follows by Proposition \ref{prop:FfFgFc}(1).

With these observations in mind, we apply Proposition \ref{prop:extendingByNormalSubgroups} with $K = G_{df}$ to get a $G$-homotopy equivalence:
\[ \Ap(G) \simeq_G  \big( \, (\Delta \Bp(H))_{\F_f}\,\big)_{\F_g} = \big( \, \Delta(H)_{\F_f}\,\big)_{\F_g}.\]

Finally, let $E\in \F_g$, so $O^{p'}(C_H(E))\in \Lie(p)$ by Proposition \ref{prop:GLSprops}(2) and Table \ref{tab:decreasingRanks}, and denote its rank by $m_E$.

{Assume first that $C_{G_{df}}(E)$ contains field automorphisms of order $p$ (i.e., $E$ commutes with some subgroup in $\F_f$).
Then these induce field automorphisms of order $p$ on $O^{p'}(C_H(E))\leq C_{G_{df}}(E)$.
Since $O^{p'}(C_H(E))$ is an $\name_p$-subgroup of $C_{G_{df}}(E)$ (see Table \ref{tab:decreasingRanks} and compare with Table \ref{tab:excludedCases}), by Theorem \ref{thm:fieldCase} the poset $\Bp(C_{G_{df}}(E))$ has the homotopy type of a wedge of spheres of dimension $m_E$.
Suppose now that $C_{G_{df}}(E)$ contains no field automorphism of order $p$ (but we still have $\F_f\neq \emptyset$).
Then $\Bp(C_{G_{df}}(E)) \simeq \Bp(C_H(E)) = \Bp(O^{p'}(C_H(E)))$ is spherical of dimension $m_E-1$.}
In any case, $m_E < n$, and therefore
\[ \Lk_{ \Delta(H)_{\F_f}}(E)  \simeq \Bp(C_{G_{df}}(E)) \]
has the homotopy type of a spherical complex of dimension strictly less than $n = \dim \Delta(H)_{\F_f}$, and $\Delta(H)_{\F_f}$ is spherical by Theorem \ref{thm:fieldCase}.
Again, by Theorem \ref{thm:MVsequence}(2), we conclude that 
\[ \big( \, \Delta(H)_{\F_f}\,\big)_{\F_g} \]
is homotopy equivalent to a wedge of spheres, and from the wedge decomposition in such a theorem we see that the dimensions described in the statement are correct.
Also, the $G$-module decomposition as in the statement follows from Theorem \ref{thm:MVsequence}(2) and Theorem \ref{thm:fieldCase}.
\end{proof}

We illustrate the previous theorem in a very particular situation.

\begin{example}
\label{ex:DnChar2}
Suppose $p = 2$, and let $G$ be a finite group with $H = F^*(G)$ of type $D_n$, $n\geq 4$, over a field of even order.
Note $H$ is an $\name_p$-subgroup of $G$.
We assume that $\F_f,\F_g$ are non-empty, and write $H = D_n(q^2)$ with $q=2^a$.
It follows that $H$ is simple of order
\[ |H| = q^{2n(n-1)}(q^{2n}-1)\prod_{i=1}^{n-1}(q^{4i}-1).\]
In particular, $G\leq \Aut(H)$ is almost simple and  $\Inndiag(H) = H \leq G$.

Now, if $x\in G$ is an order-$2$ field automorphism of $H$, then any other order-$2$ field automorphism $y\in \Aut(H)$ of $H$ is $H$-conjugate to $x$, and hence $y\in H\gen{x} \leq G$ (see Proposition \ref{prop:GLSprops}).
That is, $G$ contains every order-$2$ field automorphism of $H$, so every element of $E\in \F_g$ commutes with some element of $\F_f$.
Therefore, for any $E\in \F_g$, $C_G(E)$ contains order-$2$ field automorphisms and $C_H(E)$ has Lie rank $n-1$, i.e., $m_E^* = n-1$.
Since the elements of $\F_g$ give rise to $n$-spheres by Theorem \ref{thm:main},
$\Ap(G)$ is homotopy equivalent to a wedge of $n$-spheres.




Finally, we show that $\widetilde{\chi}(\Ap(G)) = (-1)^n \dim \widetilde{H}_n(\Ap(G),\QQ) \neq 0$.
For $E\in \F_f$, $\Ap(H E)\subseteq \Ap(G)$, where the former has non-zero homology in degree $n$ by Theorem 6.6 of \cite{PS22} (see also Corollary \ref{coro:eulerFieldCase}), and $\widetilde{H}_n(\Ap(H E),\ZZ) \subseteq \widetilde{H}_n(\Ap(G),\ZZ)$ by Theorem \ref{thm:main}. Hence $(-1)^n\widetilde{\chi}(\Ap(G)) \geq (-1)^n\widetilde{\chi}(\Ap(H E)) > 0$.
\end{example}

A straightforward computation with the ranks of the centralisers yields the following corollary on the fundamental group.
This complements the results obtained in \cite{MP}, and provides new examples of simply connected $p$-subgroup posets which can be used to compute endotrivial modules according to \cite{Jesper}.

\begin{corollary}
Let $p$ be a prime and $G$ a group with an $\name_p$-subgroup $H$.
Then $\Ap(G)$ has free fundamental group.

Moreover, if $\Ap(G)$ is connected then $\pi_1(\Ap(G))$ is non-trivial if and only if one of the following holds:
\begin{enumerate}
    \item $H$ has Lie rank $1$ and $G$ contains field automorphisms of order $p$;
    \item $H$ has Lie rank $2$ and $G$ contains no field automorphisms of order $p$;
    \item $H = A_2(2^a)$, $p = 2$, and there exists $E\in \F_g$ such that no element of $\F_f$ commutes with $E$.
\end{enumerate}
\end{corollary}

Recall that $\Ap(G)$ is disconnected if and only if $G$ has a strongly $p$-embedded subgroup \cite[Proposition 5.2]{Qui78}, and such groups have been classified during the proof of the CFSG (see, e.g., Chapter 7, Section 6 of \cite{GLS98}).

\section{Euler characteristic}
\label{sec:euler}

Recall that Quillen's conjecture states that if $O_p(G) = 1$ then $\Ap(G)$ is not contractible \cite{Qui78}.
In fact, under $O_p(G) = 1$ one usually tries to show that $\Ap(G)$ has non-zero rational homology.
For example, Quillen established this stronger conclusion for solvable groups, and this is the approach followed by M. Aschbacher and S.D. Smith \cite{AS93} to prove the conjecture for a wide class of groups.
In some special cases, it was even shown that $O_p(G) = 1$ leads to a non-zero Lefschetz module for $\Ap(G)$, which in particular implies that $\Ap(G)$ is not $\QQ$-acyclic.
For instance, this is the case for almost simple groups by the results of \cite{AK}.
So far, Quillen's conjecture remains open, and we refer to \cite{PS22} for recent developments.

In view of \cite{AK}, it is tempting to ask if for an almost simple group $G$ one has indeed $\widetilde{\chi}(\Ap(G)) \neq 0$.
For instance, this automatically holds if one of the $p$-subgroup posets $\Ap(G)$ or $\Bp(G)$ is spherical and non-contractible.
This is the case for $G$ a group of Lie type in characteristic $p$ since $\Delta \Bp(G) = \Delta(G)$ is the building of the group (see Remark \ref{rk:BpGandBuilding}).

In light of these observations, we propose:

\begin{conjecture}
\label{conj:euler}
If $G$ is an almost simple group, then $\widetilde{\chi}(\Ap(G))\neq 0$.
\end{conjecture}

To support this conjecture, in this section we show that $\widetilde{\chi}(\Ap(G))\neq 0$ for a large class of almost simple groups (see Corollary \ref{coro:eulerAlmostSimple}), starting with those such that $F^*(G)\in \Lie(r)$, for some prime $r$.

Below we show that $\widetilde{\chi}(\Ap(G)) \neq 0$ for any group $G$ containing an $\name_r$-subgroup for a prime $r\neq p$.
The proof is adapted from Proposition 8.2 of \cite{PS22}.

\begin{lemma}
\label{lm:crossCharacteristic}
Let $r$ be a prime and $G$ a group that contains an $\name_r$-subgroup.
If $p\neq r$ and $O_p(G) = 1$, then $\widetilde{\chi}(\Ap(G)) \neq 0$.
\end{lemma}

\begin{proof}
The proof is essentially the same as in Proposition 8.2 of \cite{PS22} in the case $F^*(G)\in \Lie(r)$ is simple.

Let $H$ be an $\name_r$-subgroup of $G$.
We will invoke the consequences of the Borel-Tits theorem, so we suppose first that $H\not\groupiso\PSp_4(2)$ (since such a theorem does not apply for this group), and also $O_p(G) = 1$ (in particular, $Z(H) = C_G(H)$ is a $p'$-group).
We take $Q$ to be a Sylow $r$-subgroup of $H$ and show that $\Ap(G)^Q$ is empty.
Indeed, if this holds, then an elementary counting argument shows that
\[ \widetilde{\chi}(\Ap(G)) \equiv \widetilde{\chi}(\Ap(G)^Q) \equiv -1 \pmod{r},\]
proving our assertion.

We shall prove that $\Ap(G)^Q$ is empty.
Take $E$ an elementary abelian $p$-subgroup of $G$ normalised by $Q$.
We show that $E = 1$.
Write $\overline{K}$ for the image of a subgroup $K$ of $G$ in the quotient $G/Z(H)$.
Since $Z(H)$ has order prime to $p$ and $r$, we see that $Q\groupiso \overline{Q}$ normalises $\overline{E}\groupiso E$, so it is enough to show the assertion for $Z(H) = 1$.
Then $E = 1$ follows from Proposition 8.2 of \cite{PS22} when $H$ is simple.
Hence we may assume that $H$ is as in the first column of Table \ref{tab:excludedCases}, with $Z(H) = 1$.
Moreover, the same argument given in \cite{PS22} applies to the cases $H = A_1(2)$, $A_1(3)$, ${}^2A_2(2)$, ${}^2B_2(2)$, and ${}^2G_2(3)$ since the proof of  \cite[Proposition 8.2]{PS22} invokes the conclusion ``$O_r(M)\neq 1$" of Proposition C of \cite{SST} (with $p$ there our prime $r$ here), which for these groups also holds by Lemma 2.1 of \cite{SST}, and the Borel-Tits theorem which applies to these cases as well.

Therefore, we are reduced to the cases $G = H = G_2(2) = \Aut(\PSU_3(3))$ and $G = H = {}^2F_4(2)$, with $r = 2$, and here $\Ap(G)^Q = \emptyset$ by direct computation.

Finally, $\widetilde{\chi}(\Ap(G))\neq 0$ for the case $H\groupiso \PSp_4(2)$, $r=2$, and $p=3,5$ follows by straightforward computations using that $[H,H]\groupiso \PSL_2(9)$ (see Example \ref{ex:alt6case}).
\end{proof}

\begin{example}
\label{ex:alt6case}
Let $G$ be an almost simple group such that $F^*(G) = \Alt_6 \groupiso \PSL_2(9)$.
Note that $\Bp(G) = \Bp(F^*(G))$ for odd primes $p=3,5$.

Now, if $p = 3$ then $\widetilde{\chi}(\A_3(G)) = \widetilde{\chi}(\Delta(\PSL_2(9))) = 9$.
On the other hand, since a Sylow $5$-subgroup $S$ of $\PSL_2(9)$ is cyclic of order $5$, for $p = 5$ the poset $\B_5(G)$ is disconnected with one connected component per Sylow $5$-subgroup.
Since $N_G(S)\groupiso D_{10}$, the Dihedral group of order $10$, we conclude that $\widetilde{\chi}(\A_5(G)) = |G|/|N_G(S)|-1 = 35$.

Finally, for $p = 2$, the different possibilities for $G$ and $\widetilde{\chi}(\B_p(G))$ are as follows:
If $G \groupiso \PSL_2(9)$, $\Sym_6$ or the Mathieu group $M_{10}$, then $\widetilde{\chi}(\B_2(G)) = -16$; and otherwise, $G \groupiso \PGL_2(9)$ or $\Aut(\Alt_6)$ with $\widetilde{\chi}(\B_2(G)) = -160$.
In any case, $\B_2(G)$ is spherical of dimension $1$.
\end{example}

Next, we study the equicharacteristic case $p = r$ and show that $\widetilde{\chi}(\Ap(G))\neq 0$ if $G$ contains an $\name_p$-subgroup, except possibly for some particular cases.
We will use the following expression for the Euler characteristic, derived from the conclusions of Theorem \ref{thm:main}.

\begin{corollary}
\label{coro:eulerCharFormula}
Let $p$ be a prime and $G$ a group containing an $\name_p$-subgroup $H$ of Lie rank $n$. Then:
\[ \widetilde{\chi}(\Ap(G)) = (-1)^{n-1}|H|_p - (-1)^{n-1} \sum_{\overline{E}\in \F_f/G} \frac{|G|}{|N_G(E)|} |C_H(E)|_p - \sum_{\overline{E}\in \F_g/G} \frac{|G|}{|N_G(E)|} \widetilde{\chi}(\Ap(C_{G_{df}}(E))).\]
In particular, if $r\neq p$ is a prime such that $r\mid |G|/|N_G(E)|$ for all $E\in \F_f\cup \F_g$ then $\widetilde{\chi}(\Ap(G)) \neq 0$.
\end{corollary}

\begin{proof}
Apply Theorem \ref{thm:main} and then use the fact that $$\widetilde{\chi}(\Ap(C_H(E))) = (-1)^{n-1}|C_H(E)|_p$$ for $E=1$ or $E\in \F_f$ (notice the Lie rank of $O^{p'}(C_H(E))$ always equals the Lie rank of $H$ by Lemma \ref{lm:fieldAutomorphismsAndFixedPointsBuilding}).
\end{proof}

The non-vanishing of the Euler characteristic for the case $\F_g = \emptyset$ and $G$ almost simple follows from the previous work \cite{PS22} and the results given in this article.
Below we extend this conclusion to any group $G$ with a not necessarily simple $\name_p$-subgroup.

\begin{corollary}
\label{coro:eulerFieldCase}
Let $p$ be a prime and $G$ a group containing an $\name_p$-subgroup $H$ and $\F_g = \emptyset$.
Then $\widetilde{\chi}(\Ap(G)) \neq 0$.
\end{corollary}

\begin{proof}
If $\F_f = \emptyset$, then by Corollary \ref{coro:eulerCharFormula}, $\widetilde{\chi}(\Ap(G))  = (-1)^{n-1}|H|_p\neq 0$, where $n$ denotes the Lie rank of $H$.
Thus, we may assume that $\F_f \neq\emptyset$.

By Theorem \ref{thm:fieldCase}, $\Ap(G)$ is homotopy equivalent to a wedge of spheres of dimension $n$, where $n$ is the Lie rank of $H$, so
$$\widetilde{\chi}(\Ap(G)) = (-1)^n \dim \widetilde{H}_n(\Ap(G),\QQ).$$
The result then follows by \cite[Theorem 3]{AK}, or else by the argument given below.

Let $\overline{G} = G/Z(H)$ and $\overline{H} = H/Z(H)$.
By Theorem 6.6 from \cite{PS22} we have that $\Ap(\overline{G})$ has non-zero homology in degree $n$ (note that theorem is proved for $\overline{H}$ simple, but a straightforward computation shows that it also holds when $\overline{H}$ is as in the first column of Table \ref{tab:excludedCases}).
Since $\widetilde{H}_*(\Ap(\overline{G}),\QQ) \subseteq \widetilde{H}_*(\Ap(G),\QQ)$ (cf. \cite[Lemma 0.12]{AS93}), we see that $\Ap(G)$ has non-zero homology in degree $n$.
Therefore, $\widetilde{\chi}(\Ap(G)) \neq 0$.
\end{proof}

Alternatively, with the above notation, one can also prove that if $\F_g = \emptyset$, then there is a suitable prime $r\neq 2,3,p$ such that $r\mid |G|/|N_G(E)|$ for $E\in \F_f$ and hence apply Corollary \ref{coro:eulerCharFormula}.
We will follow this approach in some cases where $\F_g \neq\emptyset$.
Note that if $\F_g\neq\emptyset$ then $p=2,3$ and $H$ is one of the groups in the first column of Table \ref{tab:decreasingRanks}.
Indeed, we may take such a prime $r$ to be a Zsigmondy prime.
Recall that Zsigmondy's theorem states that for integers $a>b>0$ and $n\geq 3$, there exists a prime $r$ such that $r\mid a^n-b^n$ and $r\nmid a^m-b^m$ for all $1\leq m<n$, if and only if $(n,a,b)\neq (6,2,1)$.

\begin{proposition}
\label{prop:equicharacteristcEuler}
Let $p$ be a prime and $G$ a group with an $\name_p$-subgroup $H$.
Assume $\F_f=\emptyset$ when $p = 2$ and $H=A_n(4^a)$ ($n\geq 2$) or $E_6(4^a)$.
Then $\widetilde{\chi}(\Ap(G)) \neq 0$.
\end{proposition}

\begin{proof}
Denote the Lie rank of $H$ by $n$.
We may assume that $\F_g \neq\emptyset$, otherwise the result follows from Corollary \ref{coro:eulerFieldCase}.
Also note that $p=2$ or $3$, and the latter occurs only if $H\cong D_4(3^a)$.

For $E\in \F_f\cup \F_g \cup \{1\}$, let $m_E$ be the Lie rank of $O^{p'}(C_H(E))$.
Then we have
\[  \widetilde{\chi}(\Ap(C_H(E))) = (-1)^{m_E-1} |C_H(E)|_p.\]
Thus $\widetilde{\chi}(\Ap(G))\neq 0$ if, for example, $\F_f = \emptyset$ and $n-1, m_E$ have the same parity for all $E\in \F_g$ (see Theorem \ref{thm:noFieldCase}).
By Table \ref{tab:decreasingRanks}, this is the case if $p = 2$, $\F_f = \emptyset$, and
\begin{equation}
    \label{eq:goodCases}
    H = \PSp_4(2^{2a+1}), \,\, D_n(2^a) \, (n\geq 4), \,\, A_{n}(2^a) \, (n\geq 2 \text{ and } n\equiv 0,1 \text{ (mod $4$)}).
\end{equation}
In case $H = D_n(2^a)$, $n\geq 4$, and $\F_f\neq\emptyset$, we have $\widetilde{\chi}(\Ap(G))\neq 0$ by Example \ref{ex:DnChar2}.

To establish the non-vanishing of the Euler characteristic in the remaining cases, namely $H$ is not as in Eq. (\ref{eq:goodCases}), and further $\F_f = \emptyset$ if $p = 2$ and $H = A_n(4^a)$ or $E_6(4^a)$, we invoke Corollary \ref{coro:eulerCharFormula} and show that there is a prime $r$ such that
\begin{equation}
\label{eq:zsigmondyCondition}
r \neq 2,3,p \quad \text{and} \quad r\mid |G|/|N_G(E)| \text{ for all }E\in \F_f\cup \F_g.
\end{equation}

We split the proof into cases according to the possibilities for $H$ as in the first column of Table \ref{tab:decreasingRanks}.
Let $E \in \F_g \cup \F_f$ or $E = 1$, and set $C_E := O^{p'}(C_{H}(E))$.
We will describe the order of $N_G(E)$ based on Table 2.2 of \cite{GLS98}.
In any case, $C_E\in \Lie(p)$, the order of $Z(C_E)$ is either a power of $2$, a divisor of $(n,q-1)$ if $C_E=A_m(q)$, or a divisor of $(n,q+1)$ if $C_E={}^2A_m(q)$.
If $z_u$ denotes the order of the centre of the universal version of $C_E$, then $|Z(C_E)|$ divides $z_u$, and we write $\mathrm{\bf z}_E:=z_u / |Z(C_E)|$.
Then we have
\begin{equation}
	|C_E| = |C_E|_p \cdot \frac{1}{\mathrm{\bf z}_E} \cdot \prod_{i\in I_E} (q_E^i - \epsilon_{E,i}),
\end{equation}
where $q_E$ is the defining field if we write $C_E = {}^{d_E}\Sigma(q_E)$, as explained in the fifth column of Table \ref{tab:decreasingRanks}, and the set $I_E$ of powers and the $\epsilon_{E,i}$ are as in the third column of Table 2.2 of \cite{GLS98}.
Moreover, $C_E$ is self-centralising and normal in $N_G(E)$.
Hence $N_G(E)/Z(C_E)$ embeds into $\Aut(C_E)$ via the quotient map
$\mathrm{\bf q}_E:N_G(E) \to \Aut(C_E)$.
We further write
\[ \mathrm{\bf d}_E := |\mathrm{\bf q}_E^{-1}(\Inndiag(C_E))|/|C_E|, \quad \text{ and } \quad \mathrm{\bf fg}_E := |N_G(E)|/(|C_E|\cdot \mathrm{\bf d}_E). \]
When $E\in \F_g\cup \F_f$ we have $N_G(E) = C_G(E).L$ with $L$ a subgroup of the cyclic group of order $p-1$, and then $\mathrm{\bf a}_E=|L|$ (so $\mathrm{\bf a}_E = 1$ or $2$).
We set $\mathrm{\bf a}_1 := 1$.

With this notation at hand, we can write
\begin{equation}
\label{eq:orderDecomposition}
|N_G(E)| = |C_E|_p \cdot \frac{1}{\mathrm{\bf z}_E} \cdot \mathrm{\bf d}_E \cdot \mathrm{\bf fg}_E \cdot \mathrm{\bf a}_E \cdot \prod_{i\in I_E} (q_E^i - \epsilon_{E,i}).
\end{equation}

We have that $\mathrm{\bf fg}_E$ divides $\mathrm{\bf fg}_1$, and also $y_E := \prod_{i\in I_E} (q_E^i - \epsilon_{E,i})$ divides $y_1 := \prod_{i\in I_1} (q_1^i - \epsilon_{1,i})$.
Finally, the terms $\mathrm{\bf z}_E$ are natural numbers which might only involve the primes $2$, $3$, and primes dividing $q_E-1$ (in the linear case), or primes dividing $q_E+1$ (in the unitary case).
We will take a prime $r$ dividing $y_1 / y_E$, and prime to $2$, $3$, $\mathrm{\bf d}_E$ and $\mathrm{\bf z}_1$.
\bigskip

\noindent
\textbf{Case 1.} If $H \cong D_4(3^a)$ and $p=3$, then there exists a prime $r$ satisfying Eq. (\ref{eq:zsigmondyCondition}).
\begin{proof}
Let $q = 3^a$ and $q' = q^{1/3}$ (if $3\mid a$).
Then
\[ y_1 = (q^4-1)(q^2-1)(q^4-1)(q^6-1),\]
\[ \mathrm{\bf z}_1 = 1, 2 \text{ or } 4,\] 
and
\[ y_E = \begin{cases}
(q'^4-1)(q'^2-1)(q'^4-1)(q'^6-1) & E\in \F_f,\\
(q^6-1)(q^2-1) & E\in \F_g, C_E \cong G_2(q),\\
(q'^8+q'^4+1)(q'^6-1)(q'^2-1) & E\in \F_g, C_E \cong {}^3D_4(q').
\end{cases}\]
For $E\in \F_f\cup \F_g$ we have
\[ y_1 / y_E = \begin{cases}
(q'^8+q'^4+1)^2 (q'^4+q'^2+1)(q'^{12}+q'^6+1) & E\in \F_f, \\
(q^4-1)^2 & E\in \F_g, C_E \cong G_2(q),\\
(q'^{2}+1)(q^{4}-1)(q^{6}-1) & E\in \F_g, C_E \cong {}^3D_4(q').
\end{cases}\]
On the other hand, $\mathrm{\bf d}_E = 1, 2, 4$ for any $E\in \F_f\cup \F_f$.
If $q'$ is not defined, so only the case of $G_2(q)$ centralisers arises, we can take $r$ to be any odd prime dividing $q^4-1$.
When $q'$ is defined, we note that $(q'^8+q'^4+1)(q'^4-1) = q^4-1$, and so $q'^8+q'^4+1$ divides $y_1/y_E$ for all $E\in \F_f\cup \F_g$.
Thus we can take any prime $r$ dividing $q'^8+q'^4+1$ (and since the latter is odd we see that $r\neq 2,3$).
In any case, we have $r\neq 2,3$ and $r\mid |G|/|N_G(E)|$ for any $E\in \F_f\cup \F_g$, i.e., $r$ satisfies Eq. (\ref{eq:zsigmondyCondition}).
\end{proof}

\noindent
\textbf{Case 2.} If $H = F_4(2^{2a+1})$ and $p=2$, then there exists a prime $r$ satisfying Eq. (\ref{eq:zsigmondyCondition}).
\begin{proof}
Here $\F_f = \emptyset$ and $\mathrm{\bf d}_E = 1 = \mathrm{\bf z}_E$ for any $E\in \F_g\cup \{1\}$.
Also, if $q = 2^{2a+1}$ and $E\in \F_g$, we have
\[ y_1/y_E = (q^{6}-1)(q^4+1)(q^3-1)(q+1).\]
Hence we can take any prime $r$ dividing $q^3-1$ (and $r\neq 2,3$ since $q$ is an odd power of $2$).
\end{proof}

\noindent
\textbf{Case 3.} If $H = E_6(2^a)$ and $p=2$, then there exists a prime $r$ satisfying Eq. (\ref{eq:zsigmondyCondition}).
\begin{proof}
Let $q = 2^{a}$ and $q' = q^{1/2}$ (if $2\mid a$).
Then $\mathrm{\bf z}_1 = 1$ or $3$, and
\[ y_1 = (q^{12}-1)(q^9-1)(q^8-1)(q^6-1)(q^5-1)(q^2-1).\]
By hypothesis we assume $\F_f = \emptyset$ if $q'$ is defined, so we only consider $E\in \F_g$, and hence
\[ y_E = \begin{cases}
(q^{12}-1)(q^8-1)(q^6-1)(q^2-1) & C_E \cong  F_4(q),\\
(q'^{12}-1)(q'^9+1)(q'^8-1)(q'^6-1)(q'^5+1)(q'^2-1) & C_E \cong {}^2E_6(q').
\end{cases}\]
Note that $|\Outdiag(F_4(q))|=1$ and $|\Outdiag({}^2E_6(q'))|=(3,q'+1)$.
Now,
\[ y_1 / y_E = \begin{cases}
(q^9-1)(q^5-1) & C_E \cong F_4(q),\\
(q^{6}+1)(q'^9-1)(q^4+1)(q^3+1)(q'^5-1)(q+1) & C_E \cong {}^2E_6(q').
\end{cases}\]
In the case $q'$ is defined, we take then a Zsigmondy prime $r$ for $q'^9-1$, and since $3\mid q'^2-1$ we have $r\neq 2,3$.
If $q'$ is not defined, so only the case $F_4(q)$ occurs, we take a Zsigmondy prime for $q^9-1$, and again $r\neq 2,3$.
\end{proof}

\noindent
\textbf{Case 4.} Let $H = A_{n}(2^a)$ and $p=2$, with $n\geq 3$.
Then there exists a prime $r$ satisfying Eq. (\ref{eq:zsigmondyCondition}).

\begin{proof}
Let $q = 2^a$, and $q' = q^{1/2}$ if $2\mid a$.
Again, in the case $q'$ is defined we are assuming $\F_f = \emptyset$.
Here $\mathrm{\bf z}_1$ is a divisor of $(n+1,q-1)$.
We have two types of automorphisms $E\in \F_g$, namely those with symplectic centraliser and those with unitary centraliser.
In the symplectic case, $|\Outdiag(\PSp_{2[(n+1)/2]}(q))| = 1$.
On the other hand, for the unitary case, we have $|\Outdiag(\PSU_{n+1}(q'))| = (n+1,q'+1)$.
This implies that $\mathrm{\bf d}_1$ divides $q-1$, and $\mathrm{\bf d}_E$ divides $q'+1$ for $E\in \F_g$.
Thus we will take a prime $r\neq 2,3$ such that in addition $r\nmid q-1$ (and hence $r\nmid q'+1$).
Here
\[ y_1 = \prod_{i=2}^{n+1}(q^i-1)\]
and
\[ y_E = \begin{cases}
\prod_{i=1}^{[(n+1)/2]}(q^{2i}-1) & C_E \cong C_{[(n+1)/2]}(q),\\
\prod_{i=2}^{n+1} (q'^i - (-1)^i) & C_E \cong {}^2A_{n}(q').
\end{cases}\]
Then
\[
y_1/y_E = \begin{cases}
 \displaystyle{ (q^3-1) \prod_{i=5 \tq i \text{ odd}}^{n+1} (q^i-1)} & C_E \cong C_{[(n+1)/2]}(q),\\
 \displaystyle{\prod_{i=2}^{n+1} \frac{(q^i-1)}{(q'^i-(-1)^i)} = \prod_{i=2}^{n+1} (q'^i+(-1)^i)} & C_E \cong {}^2A_{n}(q').     
\end{cases}
\]
If $q'$ is defined, we take a Zsigmondy prime $r\mid q'^3-1$,
so $r\nmid q'^2-1=q-1$ and $r\neq 2,3$.
In the case $q'$ is not defined, we take any prime $r\mid q^3-1$ and $r\nmid q^i-1$ for $i=1,2$, and again $r\neq 2,3$.
\end{proof}
We have treated all the cases in the first column of Table \ref{tab:decreasingRanks}, concluding the proof of this proposition.
\end{proof}

From these computations and Theorem \ref{thm:main}, we conclude:

\begin{corollary}
\label{coro:euler}
Let $p$ be a prime and $G$ a group that contains an $\name_r$-subgroup $H$, for some prime $r$.
If $H= A_{n}(4^{a})$ ($n\geq 2$) or $E_6(4^{a})$, with $p=r=2$, assume further that either $\F_f$ or $\F_g$ is empty.
Then $\widetilde{\chi}(\Ap(G)) \neq 0$.
\end{corollary}

In particular, if $G$ is an almost simple group such that $F^*(G)$ is not a sporadic group or an alternating group, then $\widetilde{\chi}(\Ap(G)) \neq 0$, except possibly for the cases excluded in Corollary \ref{coro:euler}.
Note that the case of the Tits group is covered in Example \ref{ex:titsGroup}.

Assume now that $F^*(G)$ is a sporadic simple group.
For $p$ odd, we have $\Ap(G) = \Ap(F^*(G))$.
Hence, if the $p$-rank of $F^*(G)$ is at most $2$, then $\Ap(F^*(G))$ is a graph and $\widetilde{\chi}(\Ap(G)) \neq 0$.
By Table 5.6.1 of \cite{GLS98}, we see then that, for $p\geq 5$, we have $\widetilde{\chi}(\Ap(G))\neq 0$ except possibly for $G=Co_1$ ($p=5$), $Ly$ ($p=5$), $F_5$ ($p=5$), $F_2$ ($p=5$) and $F_1$ ($p=5,7$), where the $p$-rank is at least $3$.

On the other hand, to show that $\widetilde{\chi}(\Ap(G))\neq 0$ one can argue similarly as in the proof of Proposition \ref{prop:equicharacteristcEuler} and exhibit a suitable prime $r$ that divides the index of the normaliser of any element of $\Ap(G)$ (or of any element of $\Bp(G)$).

To be more precise, suppose $X$ is a finite $G$-poset.
Then a standard counting argument with the chains of a poset shows that
\begin{equation}
\label{eq:eulerCharOrbits}
\widetilde{\chi}(X) = -1 - \sum_{x\in X} \widetilde{\chi}(X_{<x}) = -1 - \sum_{\overline{x}\in X/G} |G:\Stab_G(x)| \cdot \widetilde{\chi}(X_{<x}),
\end{equation}
where $\overline{x}$ is the $G$-orbit of $x$.
Thus, if $k := \gcd\big( |G:\Stab_G(x)| \tq \overline{x}\in X/G \big) \neq 1$, we have that $\widetilde{\chi}(X) \equiv -1\neq 0 \pmod{k}$.
Now, if $X = \Bp(G)$ for a sporadic simple group $G$, then $|G:\Stab_G(R)|$ is the index of the normaliser of a radical $p$-subgroup $R\in \Bp(G)$.
To show that $k\neq 1$ one may employ, for example, the classification of radical $p$-subgroups and their normalisers, which was concluded in 2005 by Yoshiara (see \cite{Y1, Y2}).
For instance, if $p = 7$ and $G = F^*(G) = F_1$ is the Monster sporadic group, then the order of the normalisers of radical $7$-subgroups can be read off from Theorem 7 of \cite{Y1}, and it is not hard to check that these normalisers cannot contain a Sylow $2$-subgroup.
Therefore, $r = 2$ divides the indices of these normalisers, i.e., $r\mid k$ above, and this shows that $\widetilde{\chi}(\B_7(F_1)) \neq 0$ by using Eq. (\ref{eq:eulerCharOrbits}) with $X = \B_7(G)$.

The previous approach shows that $\Ap(G)^Q$ is empty for a suitable Sylow $r$-subgroup with $r\neq p$.
One can also check this by looking at the local ranks:
After fixing the prime $p$, if there exists a prime $r$ such that the $p$-local $r$-rank of $G$ is strictly less than the $r$-rank of $G$, then a $p$-local subgroup cannot contain a Sylow $r$-subgroup of $G$, so $\Ap(G)^Q$ is empty for $Q$ a Sylow $r$-subgroup.
For example, this is the case for $p = 2$ and $r = 3$ or $5$ in many of the sporadic simple groups (cf. Tables 5.6.1 and 5.6.2 of \cite{GLS98}).

On the other hand, the argument given in \cite{AK} already shows that $\widetilde{\chi}(\Ap(G))\neq 0$ if $p\geq 5$ and $F^*(G) = \Alt_n$ for $n\geq 7$.
Also $\widetilde{\chi}(\Ap(G))\neq 0$ if $F^*(G) = \Alt_n$ with $n=5$ or $6$ since $\Alt_5 = \PSL_2(5)$ and $\Alt_6 = \PSL_2(9)$ (cf. Example \ref{ex:alt6case}).

By the preceding discussion, Corollary \ref{coro:euler} and the Classification, we conclude:

\begin{corollary}
\label{coro:eulerAlmostSimple}
If $G$ is an almost simple group then $\widetilde{\chi}(\Ap(G)) \neq 0$, except possibly in the following cases:
\begin{enumerate}
    \item $F^*(G)$ is a sporadic simple group, $p\leq 5$ and the $p$-rank of $F^*(G)$ is at least $3$;
    \item $F^*(G)\groupiso \Alt_n$ is an alternating group and $p\leq 3$;
    \item $F^*(G)\cong A_n(4^a)$ ($n\geq 2$) or $E_6(4^a)$, $p = 2$, and $\F_f$ are $\F_g$ non-empty.
\end{enumerate}
\end{corollary}

%

\appendix

\section{Homotopy theorems}
\label{sec:appendix}

In this appendix, we gather some basic tools to study the homotopy type of spaces and the behaviour of continuous maps.
Throughout this section, $G$ will denote a (not necessarily finite) group.

\begin{lemma}
\label{lm:nullHomotopicDimension}
Let $X,Y$ be CW-complexes.
Assume that $Y$ is $k$-connected for some $k\geq 0$ and that $X$ is homotopy equivalent to a CW-complex of dimension $\leq k$.
Then any continuous map $X\to Y$ is homotopic to a constant map.
\end{lemma}

\begin{proof}
Let $h:Z\to X$ be a homotopy equivalence where $Z$ is CW-complex of dimension at most $k$, and let $f:X\to Y$ be any continuous map.
Then $f$ is homotopic to a constant map if and only if $fh$ is.
Now, $\{*\}\to Y$ is a $k$-equivalence, so the composition induces a surjection $[Z,\{*\}]\to [Z,Y]$ of the homotopy classes of maps (see Whitehead's Theorem in Chapter 10, Section 3 of \cite{May}).
But $[Z,\{*\}]$ consists only of one map, so every map $Z\to Y$ is null-homotopic.
Therefore, $fh$ is null-homotopic.
\end{proof}

We will use the following variation of Quillen's fibre theorem (cf. Proposition A.1 of \cite{PS22}).

\begin{theorem}[{Quillen's fibre theorem}]
\label{thm:quillen}
Let $f:X\to Y$ be a map between posets of finite height, and let $n\geq -1$.
If $f^{-1}(Y_{\leq y}) * Y_{>y}$ is $n$-connected for all $y\in Y$ then $f$ is an $(n+1)$-equivalence.

If in addition $X,Y$ are $G$-complexes, $f$ is $G$-equivariant and $f^{-1}(Y_{\leq y}) * Y_{>y}$ is $\Stab_G(y)$-contractible for all $y\in Y$, then $f$ is a $G$-homotopy equivalence.
\end{theorem}

We will also invoke the following well-known consequence of the Hurewicz theorem and Whitehead's theorem:

\begin{theorem}
\label{thm:spherical}
Let $X$ be a CW-complex of dimension $n$.
If $X$ is simply connected and $\widetilde{H}_{m}(X,\ZZ) = 0$ for all $m < n$ then $X$ is homotopy equivalent to a (possibly empty) wedge of spheres of dimension $n$.
\end{theorem}

For equivariant homotopy types, we will use the following equivariant version of Whitehead's theorem:

\begin{theorem}
[{Equivariant Whitehead}]
\label{thm:whitehead}
Let $X,Y$ be $G$-CW-complexes together with a $G$-equivariant continuous map $f:X\to Y$.
Then $f$ is a $G$-homotopy equivalence if and only if $f:X^H\to Y^H$ is a homotopy equivalence for every subgroup $H\leq G$.
\end{theorem}

Recall that if $X,Y$ are $G$-posets and $f,g:X\to Y$ are $G$-equivariant order-preserving maps such that $f(x)\leq g(x)$ for all $x\in X$, then $f,g$ induce $G$-homotopy equivalent maps in the geometric realisations.

\begin{proposition}
\label{prop:linksHighlyConnected}
Let $K$ be a finite dimensional $G$-complex, $L\leq K$ a full $G$-invariant subcomplex, and $n\geq -1$.
Assume that for all $\sigma \in K\setminus L$, the link $\Lk_L(\sigma)$ is $(n-|\sigma|)$-connected (resp. $\Stab_G(\sigma)$-contractible).
Then the inclusion $L\hookrightarrow K$ is an $n$-equivalence (resp. a $G$-equivalence).
\end{proposition}

\begin{proof}
Regard $K$ and $L$ as posets via their face posets, and we remove the empty simplex (otherwise we get contractible posets).
Let $\N$ be the subposet of $K$ consisting of simplices containing some vertex in $L$.
Then the map $r:\N\to L$ that takes $\sigma\in \N$ to the non-empty simplex $\sigma \cap L := \{v\in \sigma\tq v\in L\} \in L$ is well-defined, $G$-equivariant and order-preserving (here we are using that $L$ is a full subcomplex).
Moreover, if $i:L\hookrightarrow \N$ is the inclusion then $ri$ is the identity of $L$ and $ir(\sigma)\leq \sigma$, so $r$ and $i$ are $G$-homotopy equivalences.
Now, the inclusion $L\hookrightarrow K$ factorises through $\N$ at the face poset level, so it is enough to show that $\N\hookrightarrow K$ is an $n$-equivalence (resp. a $G$-equivalence).

Note that $K\setminus \N$ is exactly the face poset of $K\setminus L$.
For $\sigma \in K\setminus L$, we have $\N_{\geq \sigma} = \{\tau\in K \tq \tau\cap L \neq\emptyset, \tau \geq \sigma\}$, and for $\tau \in \N_{\geq \sigma}$ we have $r(\tau)\in \Lk_L(\sigma)$.
The map $r$ above restricts therefore to a map $r: \N_{\geq \sigma} \to \Lk_L(\sigma)$ which is also a $\Stab_G(\sigma)$-homotopy equivalence with inverse given by $\tau \in \Lk_L(\sigma)\mapsto \tau\cup \sigma \in \N_{\geq \sigma}$.
This shows that
\[ \N_{\geq \sigma} * K_{<\sigma} \simeq_{\Stab_G(\sigma)} \Lk_L(\sigma) * S^{|\sigma|-2},\]
which is $(n-|\sigma|)+(|\sigma|-3)+2 = (n-1)$-connected (resp. $\Stab_G(\sigma)$-contractible).
By Quillen's fibre Theorem \ref{thm:quillen} we conclude that $\N\hookrightarrow K$ is an $n$-equivalence (resp. a $G$-homotopy equivalence), so the composition $L\hookrightarrow \N\hookrightarrow K$ is an $n$-equivalence (resp. a $G$-homotopy equivalence). 
\end{proof}

\begin{lemma}
[{``Gluing lemma"}]
\label{lm:gluing}
Suppose we have a commutative diagram of topological spaces with continuous maps
\[ \xymatrix{
X \ar[d] & A \ar[l]_f \ar[r] \ar[d] & Y \ar[d]\\
X' & A' \ar[l]_{f'} \ar[r] & Y'
}\]
such that the vertical arrows are homotopy equivalences and $f,f'$ are cofibrations.
Then the map induced on the pushouts $X\cup_A Y\to X'\cup_{A'}Y'$ is a homotopy equivalence.
\end{lemma}

\begin{proof}
See Lemma 2.1.3 of \cite{May2} or 7.5.7 of \cite{Brown}.
\end{proof}

\bigskip

\noindent
\textbf{Declarations of interest:} none.


\begin{thebibliography}{}
\bibitem[AB]{AB} P. Abramenko, K.~S. Brown. {\it Buildings}, Graduate Texts in Mathematics, 248, Springer, New York, 2008.

\bibitem[AK]{AK}
M. Aschbacher, P.B. Kleidman. \textit{On a conjecture of Quillen and a lemma of Robinson}, Arch. Math. (Basel) \textbf{55} (1990), no. 3, 209-217.

\bibitem[AS]{AS93}
M. Aschbacher, S.D. Smith. \textit{On Quillen's conjecture for the $p$-groups complex}, Ann. of Math. (2) {\bf 137} (1993), no.~3, 473--529.

\bibitem[BT]{BorelTits}
A. Borel, J. Tits. \textit{Groupes r\'{e}ductifs}, Inst. Hautes \'{E}tudes Sci. Publ. Math. No. 27 (1965), 55--150.

\bibitem[Bo]{Bouc}
S. Bouc. \textit{Homologie de certains ensembles ordonn\'{e}s}, C. R. Acad. Sci. Paris S\'{e}r. I Math. {\bf 299} (1984), no.~2, 49--52.

\bibitem[Br]{Brown}
R. Brown. {\it Topology}, second edition, Ellis Horwood Series: Mathematics and its Applications, Horwood, Chichester, 1988.


\bibitem[FPSC]{FPSC}
X. Fernández, K. I. Piterman, I. Sadofschi Costa. ‘Posets - posets and finite spaces - a GAP package, version 1.0.2’, 2020, 
\href{https://doi.org/10.5281/zenodo.3851453}{https://doi.org/10.5281/zenodo.3851453}.


\bibitem[GAP22]{Gap}
The GAP Group, GAP -- Groups, Algorithms, and Programming, Version 4.12.2; 2022. (https://www.gap-system.org)


\bibitem[GLS]{GLS98}
D. Gorenstein, R. Lyons, R. Solomon. \textit{The classification of the finite simple groups. Number 3. Part I. Chapter A}, Mathematical Surveys and Monographs, vol. 40, American Mathematical Society, Providence, RI, 1998, Almost simple K-groups.

\bibitem[G]{Jesper}
J. Grodal. \textit{Endotrivial modules for finite groups via homotopy theory}, J. Amer. Math. Soc. {\bf 36} (2023), no.~1, 177--250.

\bibitem[H]{Hatcher}
A.E. Hatcher. \textit{Algebraic topology}, Cambridge Univ. Press, Cambridge, 2002.

\bibitem[MT]{MalleTesterman}
G. Malle, D. Testerman.
{\it Linear algebraic groups and finite groups of Lie type}, Cambridge Studies in Advanced Mathematics, 133, Cambridge Univ. Press, Cambridge, 2011.

\bibitem[M]{May}
J.P. May. \textit{A concise course in algebraic topology}, Chicago Lectures in Mathematics, Univ. Chicago Press, Chicago, IL, 1999.

\bibitem[MaP]{May2}
J.P. May, K. Ponto. \textit{More concise algebraic topology}, Chicago Lectures in Mathematics, Univ. Chicago Press, Chicago, IL, 2012.

\bibitem[MiP]{MP}
E.G. Minian, K.I. Piterman.
\textit{The fundamental group of the  $p$-subgroup complex},
J. Lond. Math. Soc. (2) \textbf{103} (2021), no. 2, 449–469.

\bibitem[PS]{PS22}
K.I. Piterman, S.D. Smith. \textit{Some results on Quillen's Conjecture via equivalent-poset techniques}. J. Comb. Algebra (2024), published online first.
\href{https://doi.org/10.4171/jca/95}{https://doi.org/10.4171/jca/95}.


\bibitem[PW]{PW}
J. Pulkus, V. Welker.\textit{ On the homotopy type of the $p$-subgroup complex for finite solvable groups}, J. Austral. Math. Soc. Ser. A {\bf 69} (2000), no.~2, 212--228. 

\bibitem[Q]{Qui78} D. Quillen. \textit{Homotopy properties of the poset of nontrivial $p$-subgroups of a group}, Adv. Math. \textbf{28} (1978), 101--128.


\bibitem[SW]{SW}
Y. Segev, P.~J. Webb. \textit{Extensions of $G$-posets and Quillen's complex}, J. Austral. Math. Soc. Ser. A {\bf 57} (1994), no.~1, 60--75.

\bibitem[SST]{SST}
G.M. Seitz, R. Solomon and A. Turull. \textit{Chains of subgroups in groups of Lie type. II}, J. London Math. Soc.
(2) \textbf{42} (1990), no. 1, 93–100.


\bibitem[Sh]{Shareshian}
J. Shareshian. \textit{Hypergraph matching complexes and Quillen complexes of symmetric groups}, J. Combin. Theory Ser. A {\bf 106} (2004), no.~2, 299--314. 

\bibitem[St1]{Steinberg}
R. Steinberg. {\it Lectures on Chevalley groups}, Yale University, New Haven, CT, 1968.

\bibitem[St2]{St2}
R. Steinberg. {\it Endomorphisms of linear algebraic groups}, Memoirs Amer. Math. Soc. \textbf{80} (1968). 

\bibitem[T]{Tits}
J. Tits.
{\it Buildings of spherical type and finite BN-pairs}, Lecture Notes in Mathematics, Vol. 386, Springer, Berlin, 1974.

\bibitem[W]{Webb}
P.~J. Webb. \textit{Subgroup complexes}, in \textit{The Arcata Conference on Representations of Finite Groups (Arcata, Calif., 1986)}, 349--365, Proc. Sympos. Pure Math., 47, Part 1, Amer. Math. Soc., Providence, RI.

\bibitem[Y1]{Y1}
S. Yoshiara. \textit{Odd radical subgroups of some sporadic simple groups}, J. Algebra {\bf 291} (2005), no.~1, 90--107.

\bibitem[Y2]{Y2}
S. Yoshiara. \textit{Radical $2$-subgroups of the Monster and the Baby Monster}, J. Algebra {\bf 287} (2005), no.~1, 123--139. 
\end{thebibliography}
\end{document}